\documentclass[oneside, reqno] {amsart}

\usepackage{xypic}
\input xy
\xyoption{all}
\usepackage{epsfig}
\usepackage{amsthm}
\usepackage{amssymb}
\usepackage{amsmath}
\usepackage{amscd}
\usepackage{color}
\usepackage[T1]{fontenc}
\usepackage{upgreek}
\usepackage[font=scriptsize]{caption}
\usepackage{wrapfig}
\usepackage{graphicx}
\usepackage{subfigure}

\newcommand{\bg}{\begin{equation}}
\newcommand{\ed}{\end{equation}}
\newcommand{\bga}{\begin{eqnarray}}
\newcommand{\eda}{\end{eqnarray}}

\def\cbdu{\par{\raggedleft$\Box$\par}}

\newtheorem {Theorem}  {Theorem}

\numberwithin{Theorem}{section}

\newtheorem {Lemma}[Theorem]  {Lemma}

\theoremstyle{definition}

\theoremstyle{remark}

\expandafter\chardef\csname pre amssym.def
at\endcsname=\the\catcode`\@ \catcode`\@=11
\def\undefine#1{\let#1\undefined}
\def\newsymbol#1#2#3#4#5{\let\next@\relax
 \ifnum#2=\@ne\let\next@\msafam@\else
 \ifnum#2=\tw@\let\next@\msbfam@\fi\fi
 \mathchardef#1="#3\next@#4#5}
\def\mathhexbox@#1#2#3{\relax
 \ifmmode\mathpalette{}{\m@th\mathchar"#1#2#3}%
 \else\leavevmode\hbox{$\m@th\mathchar"#1#2#3$}\fi}
\def\hexnumber@#1{\ifcase#1 0\or 1\or 2\or 3\or 4\or 5\or 6\or 7\or 8\or
 9\or A\or B\or C\or D\or E\or F\fi}

\font\teneufm=eufm10 \font\seveneufm=eufm7 \font\fiveeufm=eufm5
\newfam\eufmfam
\textfont\eufmfam=\teneufm \scriptfont\eufmfam=\seveneufm
\scriptscriptfont\eufmfam=\fiveeufm

\catcode`\@=\csname pre amssym.def at\endcsname

\newcounter{remark}
\newcommand{\R}{\mathbf{R}}

\renewcommand{\div}{\mbox{div}}

\def  \R   {{\mathbb R}}

\def  \T   {{\mathbb T}}

\def  \12  {{\frac{1}{2}}}

\def\build#1_#2^#3{\mathrel{\mathop{\kern 0pt#1}\limits_{#2}^{#3}}}

\numberwithin{equation}{section}

\begin{document}
%\currannalsline{0}{2006}

\title[Determining wavenumbers for hall and electron MHD]{Determining wavenumbers for Hall and electron magnetohydrodynamics turbulence}

%\author{hello}

\author [Hassan]{Hassan Babaei}

\address{Department of Mathematics, Statistics and Computer Science, University of Illinois at Chicago, Chicago, IL 60607, USA}
\email{hbabae2@uic.edu} 

\author [Mimi]{Mimi Dai}

\address{Department of Mathematics, Statistics and Computer Science, University of Illinois at Chicago, Chicago, IL 60607, USA}
\email{mdai@uic.edu}

\thanks{H. Babaei is partially supported by the UIC’s Victor Twersky Memorial Scholarship. M.Dai is partially supported by the NSF grant DMS–2308208 and Simons Foundation.}

\begin{abstract}
In turbulent flows, the Kolmogorov wavenumber characterizes the smallest scales at which viscous effects dominate. A mathematical analogue of this notion first introduced by Foias and Prodi \cite{FP}, a determining wavenumber--quantifies the minimal set of modes that uniquely determine the long-time behavior of solutions. Extending this framework from the Navier--Stokes equations to magnetized plasma models, we focus on the Hall-MHD and Electron-MHD turbulence in sub-ion and dissipation ranges.  

We prove existence of time-dependent determining wavenumbers for weak solutions of the Hall- and electron-MHD, improving upon previous results that were not optimal and lacked any comparison with phenomenological dissipation scales. Under explicit scale-localized intermittency assumptions, we show that their time averages are bounded above by Kolmogorov-like dissipation wavenumbers predicted by phenomenological studies of plasma turbulence. For strong electron-MHD solutions, we also establish a uniform bound on the magnetic determining wavenumber from Besov regularity.

\bigskip

KEY WORDS: magnetohydrodynamics; Hall effect; dissipation wavenumber; turbulence theory.  

%\hspace{0.02cm}CLASSIFICATION CODE: 35Q35, 76B03, 76D09, 76E25, 76W05.
\end{abstract}

\maketitle

\section{Introduction}

%\medskip

%\subsection{Overview}

\subsection{Hall-MHD and Electron-MHD Models}

The Hall magnetohydrodynamic (Hall-MHD) equations arise as an extension of the
classical magnetofluid description when additional two–fluid effects become
important. In particular, at spatial scales comparable to the \emph{ion inertial
length} $d_i$, the motions of electrons and ions decouple and the standard Ohm’s
law must be modified to incorporate the \emph{Hall current}. This leads to the
incompressible Hall-MHD system
\begin{equation}
\begin{split}\label{hall-mhd}
&\partial_t u + u\cdot\nabla u + b\cdot\nabla b +\nabla p
= \nu \Delta u , \\
&\partial_t b + u\cdot\nabla b - b\cdot\nabla u\,
 +\, d_i \,\nabla \times \big((\nabla \times b)\times b\big)
= \mu \Delta b, \\
&\nabla \cdot u = 0, \qquad \nabla\cdot b =0 .
\end{split}
\end{equation}

Here $u(x,t)$ denotes the velocity field, $b(x,t)$ the magnetic field, $p$ the
pressure, and $\nu,\mu>0$ represent the viscosity and magnetic diffusivity.
The parameter $d_i$ characterizes the ion inertial scale at which Hall effects
become significant.

When $d_i=0$, the Hall term disappears and the system reduces to the classical
magnetohydrodynamic equations. In contrast, when $d_i\neq0$, additional
dispersive and nonlinear effects appear due to the coupling between the
magnetic field and current density. These effects are known to play a central
role in several plasma processes such as magnetic reconnection and small-scale
magnetic turbulence observed in laboratory plasmas, the Earth's magnetosphere,
and astrophysical environments.

A related model arises in the opposite regime where ions form a stationary
neutralizing background and the magnetic field evolution is governed entirely
by the electron flow. In this limit one obtains the
\emph{Electron-MHD (EMHD)} equation
\begin{equation}
\begin{split}
&\partial_t b + d_i\,\nabla\times\big((\nabla\times b)\times b\big)
= \mu \Delta b,\\
& \nabla\cdot b =0 .
\end{split}
\end{equation}

This model describes magnetic dynamics at scales smaller than the ion inertial
length and is dominated by dispersive whistler-wave interactions. EMHD
turbulence plays an important role in modeling the sub-ion-scale dynamics of
collisionless plasmas such as the solar wind and magnetospheric plasmas.

Phenomenological studies of EMHD turbulence predict a magnetic energy cascade
governed by nonlinear interactions of whistler modes. Dimensional arguments
suggest that the magnetic energy spectrum obeys the scaling law \(-10/3\), which differs from the classical Kolmogorov spectrum observed in
hydrodynamic turbulence. This spectral law has been supported by numerical simulations and observational evidence in plasma turbulence. 
%environments, and itcorresponds to a cascade process that transfers magnetic energy toward small scales until dissipation becomes dominant.

The Hall-MHD system exhibits a richer phenomenology due to the simultaneous
presence of velocity and magnetic fields and the Hall current interaction.
Plasma turbulence in Hall-MHD is typically characterized by three distinct
regimes: an ion-inertial range at large scales, a sub-ion range where Hall
effects dominate, and a final dissipation range at the smallest scales.
Each of these regimes is associated with different scaling behaviors for
the energy spectrum. In particular, within the sub-ion regime the dynamics
resemble those of EMHD turbulence and similar spectral scalings emerge.
These phenomenological predictions indicate the presence of characteristic
dissipation wavenumbers that generalize the classical Kolmogorov scale to
magnetized plasma systems Ref.~\cite{4}.

\subsection{Determining Wavenumbers and Degrees of Freedom}

Kolmogorov’s theory of turbulence predicts that in three-dimensional flows
most of the energy is contained in Fourier modes below a critical wavenumber,
known as the \emph{Kolmogorov dissipation wavenumber}, beyond which viscous
effects dominate. This idea suggests that turbulent flows possess an
effective number of dynamically relevant degrees of freedom determined by
the dissipation scale.

A rigorous mathematical analogue of this concept arises in the theory of
\emph{determining modes}. The notion of determining modes was introduced by
Foias and Prodi for the two-dimensional Navier–Stokes equations and later
developed extensively in the study of dissipative dynamical systems.
Roughly speaking, determining modes represent finitely many Fourier modes
whose knowledge uniquely determines the asymptotic behavior of solutions
as time tends to infinity. Consequently, they provide a mathematical
quantification of the finite dimensionality of turbulent dynamics.

In recent work \cite{1,2}, Cheskidov, Dai, and collaborators established the existence of a
\emph{determining wavenumber} for the 3D Navier–Stokes
equations. Their framework introduces a generalized Kolmogorov dissipation
wavenumber \( \kappa_d :=
\left(
\frac{\varepsilon}{\nu^3}
\right)^{\frac{1}{d+1}},
\)
where $\varepsilon = \nu \langle \|\nabla u\|_2^2 \rangle$ denotes the
average energy dissipation rate and $d\in[0,3]$ represents the
\emph{intermittency dimension} of the flow. The case $d=3$ corresponds to
the classical Kolmogorov scaling without intermittency, while smaller
values of $d$ capture deviations arising from intermittent turbulent
structures.

In this framework the determining wavenumber $\Lambda_u(t)$ separates dynamically active large scales from smaller scales that are slaved to
them. The time-averaged determining wavenumber is controlled by the generalized
Kolmogorov dissipation number, showing that the determining-wavenumber
framework aligns quantitatively with turbulence phenomenology and provides a
rigorous description of the effective degrees of freedom in turbulent
flows.

\subsection{Motivation and Contribution of the Present Work}

The natural question then arises whether a similar correspondence between
determining wavenumbers and phenomenological dissipation scales can be
established for magnetized plasma systems such as the Hall-MHD and EMHD
equations.

Phenomenological and numerical studies of plasma turbulence suggest that Kolmogorov-like dissipation wavenumbers also arise in these systems. In the context of EMHD turbulence it has been conjectured that the dissipation scale is given by
\[
\kappa^{\delta_b}_e :=
\left(
\frac{\varepsilon_b}{\mu^3}
\right)^{\frac{1}{\delta_b-1}},
\]
where, in the dyadic normalization used later in Section~5,
\[
\varepsilon_b:=\mu \lambda_0^{\delta_b}\left\langle
\sum_{q\ge -1}\lambda_q^2\|b_q\|_2^2
\right\rangle,
\]
and $\delta_b\in[0,3]$ represents the intermittency dimension of the
magnetic field. The case $\delta_b=3$ corresponds to the Kolmogorov-like
scaling observed in numerical studies of plasma turbulence.

In this paper we establish a rigorous mathematical framework connecting
these phenomenological dissipation scales with determining wavenumbers
for both EMHD and Hall-MHD systems. We introduce new definitions of the
determining wavenumbers $\Lambda_{u,r}$ and $\Lambda_{b,r}$ associated
with the velocity and magnetic fields, inspired by the refined
construction developed for the Navier–Stokes equations in
Refs.~\cite{1,2}. This choice is natural since the velocity equation in
the Hall-MHD system retains the same scaling structure as the
Navier–Stokes equation.

For the EMHD system we prove the existence of an optimal magnetic
determining wavenumber $\Lambda_{b,r}$ and, under an explicit scale-localized
intermittency assumption, show that its time average satisfies the upper bound
\[
\langle \Lambda_{b,r} \rangle \lesssim \lambda_0+\kappa_e^{\delta_b},
\]
which agrees with the dissipation scale predicted by EMHD turbulence
phenomenology.

The Hall-MHD system presents a more complex scenario due to the richer
nonlinear coupling between velocity and magnetic fields. Plasma
turbulence in this system involves multiple scaling regimes including
the ion-inertial, sub-ion, and dissipation ranges. Our analysis focuses
on the sub-ion and dissipation regimes where the scaling properties
resemble those of EMHD turbulence. We establish the existence of
determining wavenumbers for Hall-MHD and, again under scale-localized
intermittency assumptions, prove the upper bounds
\[
\langle \Lambda_{u,r} \rangle \lesssim \lambda_0+\kappa_u^{\delta_u},
\qquad
\langle \Lambda_{b,r} \rangle \lesssim \lambda_0+\kappa_e^{\delta_b},
\]
where \(\kappa_u^{\delta_u}\) is the Navier--Stokes-type dissipation number
associated with the velocity intermittency dimension \(\delta_u\).

These results provide a rigorous mathematical counterpart to the
phenomenological description of plasma turbulence and extend the
determining modes theory beyond the classical hydrodynamic setting.

\subsection{Outline of the Paper}

The remainder of this paper is organized as follows. In Section~2 we
review the Hall-MHD and Electron-MHD models and introduce the functional
framework used throughout the paper. In Section~3 we define the
determining wavenumbers for these systems following the Littlewood–Paley
framework developed for the Navier–Stokes equations. In Section~4 we
derive estimates for these quantities and establish their boundedness
and scaling properties. In Section~5 we compare the averaged determining
wavenumbers with phenomenological dissipation wavenumbers predicted by
Hall-MHD and EMHD turbulence theories. Finally, in Section \ref{sec-uniform-bound}, we show that a criterion guarantees that the determining wavenumber stays uniformly bounded.
%we discuss the implications of these results and outline possible extensions to other magnetized plasma models. 

\section{Preliminaries}
\label{sec:pre}

\subsection{Notation}
\label{sec:notation}

Throughout the paper, the notation $A \lesssim B$ indicates an estimate of the form $A \leq C B$ for some positive constant $C$ independent of the relevant parameters. 

For a tempered distribution $u$, we denote its Fourier transform by $\mathcal{F}u=\hat{u}$ and the inverse Fourier transform by $\mathcal{F}^{-1}u=\check{u}$. 

For convenience, the $L^p$-norm $\|\cdot\|_{L^p}$ will occasionally be abbreviated as $\|\cdot\|_p$. The symbol $H^s$ denotes the standard $L^2$-based Sobolev space of order $s$.

\subsection{Well-posedness results for system (\ref{hall-mhd})}

Before turning to determining wavenumbers, we record the solution classes for
the Hall-MHD system that are used implicitly in the later analysis. We only
need a brief summary of the standard well-posedness theory on the
three-dimensional torus.

At the level of finite energy, global Leray--Hopf type weak solutions are
available on $\T^3$; see \cite{ADFL}. This is the Hall-MHD counterpart of the
classical global weak existence theory for the Navier--Stokes equations.

\begin{Theorem}[Leray--Hopf type solutions]
Let the initial data satisfy
\[
u_0,b_0\in L^2(\T^3),
\qquad
\nabla\cdot u_0=\nabla\cdot b_0=0.
\]
Then system \eqref{hall-mhd} admits a global weak solution \((u,b)\) with
\[
(u,b) \in L^\infty \big(0,T; (L^2(\T^3))^2\big)
\cap L^2\big(0,T; (H^1(\T^3))^2\big).
\]
In addition, the energy inequality
\[
\frac{1}{2}\frac{\mathrm{d}}{\mathrm{d}t}
\big(\|u\|_{L^2}^2+\|b\|_{L^2}^2\big)
+\nu \|\nabla u\|_{L^2}^2
+\mu \|\nabla b\|_{L^2}^2\leq 0.
\]
\end{Theorem}

For smoother divergence-free data one has local classical solvability, together
with global continuation for sufficiently small initial data; see \cite{CDL}.

\begin{Theorem}[Strong solutions]
Let \(s>\frac{5}{2}\) be an integer, and let
\(u_0,b_0\in H^s(\T^3)\) satisfy
\(\nabla\cdot u_0=\nabla\cdot b_0=0\). Then:

\medskip
\noindent
(i) There exists
\(T=T(\|u_0\|_{H^s},\|b_0\|_{H^s})>0\)
such that \((u_0,b_0)\) generates a classical solution satisfying
\[
(u,b)\in L^\infty\big(0,T;(H^s(\T^3))^2\big).
\]

\noindent
(ii) There exists a constant \(\epsilon=\epsilon(\nu,s)>0\) such that if
\[
\|u_0\|_{H^s}+\|b_0\|_{H^s}<\epsilon,
\]
then the corresponding strong solution extends globally and obeys
\[
(u,b)\in L^\infty\big(0,\infty;(H^s(\T^3))^2\big).
\]
\end{Theorem}

For the long-time arguments behind determining modes, one also needs a
Prodi--Serrin type continuation criterion. We will not use its full statement
explicitly, but we note that \cite{CL} gives such a blow-up criterion in terms
of space-time integrability conditions on the velocity and magnetic field.

%\begin{Theorem}[Prodi--Serrin type regularity criterion]
%\label{rgc}
%Let $s>\frac{5}{2}$ be an integer and assume $u_0,b_0\in H^s(\T^3)$ with $\nabla\cdot u_0=\nabla\cdot b_0=0$. Let $T^*<\infty$ denote the first blow-up time of the corresponding classical solution to system (\ref{hmhd1})--(\ref{hmhd3}). Then
%\[
%\limsup_{t\nearrow T^*}
%\big(\|u(t)\|_{H^s}^2+\|b(t)\|_{H^s}^2\big)
%=\infty
%\]
%if and only if
%\[
%\|u\|_{L^q(0,T^*;L^p(\T^3))}
%+
%\|\nabla b\|_{L^\gamma(0,T^*;L^\beta(\T^3))}
%=\infty,
%\]
%where the exponents satisfy
%\[
%\frac{3}{p}+\frac{2}{q}\le1,
%\qquad
%\frac{3}{\beta}+\frac{2}{\gamma}\le1,
%\qquad
%p,\beta\in(3,\infty].
%\]
%\end{Theorem}

\subsection{Littlewood--Paley decomposition}

We next fix the dyadic frequency decomposition that will be used throughout the
paper. This is the standard Littlewood--Paley framework on the torus, written
here in the notation needed for the later energy estimates.

Let $\lambda_q=2^q$ for $q\in\mathbb{Z}$. Choose a radial cutoff function $\chi\in C_0^\infty(\R^n)$ satisfying
\[
\chi(\xi)=
\begin{cases}
1, & |\xi|\le\frac{3}{4},\\
0, & |\xi|\ge1 .
\end{cases}
\]
Define
\[
\varphi(\xi)=\chi(\xi/2)-\chi(\xi),
\qquad
\varphi_q(\xi)=
\begin{cases}
\varphi(\lambda_q^{-1}\xi), & q\ge0,\\
\chi(\xi), & q=-1 .
\end{cases}
\]

Given a vector field \(u\in\mathcal{S}'(\T^n)\), we define its dyadic blocks by
\[
\Delta_q u=u_q:=\sum_{k\in\mathbb{Z}^n}
\varphi_q(k)\hat{u}(k)e^{i2\pi k\cdot x},
\]
where \(\hat u(k)\) is the \(k\)-th Fourier coefficient of \(u\). In
particular, the very low frequencies are collected in the block \(u_{-1}\).

The field \(u\) is then recovered, in the sense of distributions, from the
series
\[
u=\sum_{q=-1}^\infty u_q .
\]

For later convenience we also write
\[
u_{\le Q}:=\sum_{q=-1}^{Q}u_q,
\qquad
u_{(P,Q]}:=\sum_{q=P+1}^{Q}u_q,
\qquad
\tilde{u}_q:=\sum_{|p-q|\le1}u_p .
\]

In this notation, the Sobolev norm admits the standard dyadic characterization
\[
\|u\|_{H^s}
=
\left(
\sum_{q\ge-1}
\lambda_q^{2s}\|u_q\|_2^2
\right)^{1/2}.
\]

We will repeatedly use the Bernstein estimate associated with this frequency
localization.

\begin{Lemma}
\label{brn}
Let \(n\) be the spatial dimension and suppose \(s\ge r\ge1\). Then
\[
\|u_q\|_r
\lesssim
\lambda_q^{\,n(\frac{1}{r}-\frac{1}{s})}
\|u_q\|_s .
\]
\end{Lemma}

\subsection{Bony's paraproduct and commutator estimates}

We now recall the paradifferential decomposition used to organize nonlinear
interactions. If \(u\) and \(v\) are distributions, then formally one may write
\[
uv=\sum_{p,q\ge-1}u_p v_q .
\]
Bony's decomposition separates this product into low--high, high--low, and
comparable-frequency contributions:
\[
uv=
\sum_{q\ge-1}u_{\le q-2}v_q
+
\sum_{q\ge-1}u_q v_{\le q-2}
+
\sum_{q\ge-1}\tilde{u}_q v_q .
\]

This splitting is the basic tool for estimating the transport and Hall terms in
dyadic form.

For later use we introduce the commutators associated with the convection and
Hall operators:
\begin{equation}
[\Delta_q,u_{\le p-2}\cdot\nabla]v_p
=
\Delta_q(u_{\le p-2}\cdot\nabla v_p)
-
u_{\le p-2}\cdot\nabla\Delta_q v_p,
\label{cm1}
\end{equation}
and
\begin{equation}
[\Delta_q,b_{\le p-2}\times\nabla\times]h_p
=
\Delta_q(b_{\le p-2}\times(\nabla\times h_p))
-
b_{\le p-2}\times(\nabla\times\Delta_q h_p).
\label{cm2}
\end{equation}

Together with the divergence-free constraint, these commutators capture the
cancellation structure that appears repeatedly in the proofs.

The next estimate for the transport commutator is standard; see \cite{BCD}.

\begin{Lemma}
\label{cmest}
Let $\frac{1}{r_1}=\frac{1}{r_2}+\frac{1}{r_3}$. Then
\[
\|[\Delta_q,u_{\le p-2}\cdot\nabla]v_p\|_{r_1}
\lesssim
\|v_p\|_{r_2}
\sum_{p'\le p-2}\lambda_{p'}\|u_{p'}\|_{r_3}.
\]
\end{Lemma}

The Hall commutator satisfies the analogous bound.

\begin{Lemma}
\label{cmmes}
Let $\frac{1}{r_1}=\frac{1}{r_2}+\frac{1}{r_3}$.
Assume $\nabla\cdot b_{\le p-2}=0$. Then
\[
\|[\Delta_q,b_{\le p-2}\times\nabla\times]h_p\|_{r_1}
\lesssim
\|h_p\|_{r_2}
\sum_{p'\le p-2}\lambda_{p'}\|b_{p'}\|_{r_3} .
\]
\end{Lemma}

Additional background on Littlewood--Paley theory and paradifferential calculus can
be found in the monograph of Bahouri, Chemin and Danchin \cite{BCD}.

\vspace{0.5cm}

\section{Existence of Determining Wavenumbers for Hall- and Electron-MHD}

\subsection{Electron-MHD}
First note that the equation for electron MHD is given by,
\begin{equation}\label{emhd}
\begin{split}
b_t+ \nabla\times ((\nabla\times b)\times b)=&\ \mu\Delta b,\\
\nabla\cdot b=&\ 0.
\end{split}
\end{equation}
When $n=2$, it is understood that
\[b=b(x,y,t)=\left(b_1(x,y,t), b_2(x,y,t), b_3(x,y,t) \right).\]
We will establish the existence of determining wavenumber for (\ref{emhd}) following the framework developed in \cite{1, 2}  via using wavenumber splitting. 

According to the scaling of (\ref{emhd}), we define the wavenumber 
\begin{equation}\label{wave}
\Lambda_{b,r} (t)=\min \{\lambda_q: (L\lambda _{p-q})^{\delta}\lambda_p^{\frac{n}{r}}\|B_p\|_{L^r}<c_r \mu, \forall \ p>q; \ \mbox{and} \ \ \lambda_q^{-1}\, \|\nabla B_{\leq q}\|_{L^\infty}<c_r\mu, q\in \mathbb N\}
\end{equation}
for a small constant $c_r>0$. Note both quantities $\lambda_p^{\frac{n}{r}}\|B_p\|_{L^r}$ and $\lambda_q^{-1}\, \|\nabla B_{\leq q}\|_{L^\infty}$ are scaling invariant. We prove

\begin{Theorem}\label{thm-det}
Let $b^{(1)}(t)$ and $b^{(2)}(t)$ be two weak solutions of (\ref{emhd}). Let $\Lambda_{b^{(1)}}(t)$ and $\Lambda_{b^{(2)}}(t)$ be the wavenumber defined for $b^{(1)}(t)$ and $b^{(2)}(t)$ respectively as in (\ref{wave}). Denote 
\[\Lambda (t)=\max\{\Lambda_{b^{(1)}}(t), \Lambda_{b^{(2)}}(t)\} \] and let $Q(t)$ be such that $\Lambda(t)=\lambda_{Q(t)}$. 
Assume \[ \left(b^{(1)}-b^{(2)}\right)|_{\leq \Lambda}=0. \]
Then we have 
\[\lim_{t\to\infty} \|b^{(1)}-b^{(2)}\|_{H^s}=0, \ \ \ \mbox{for} \ \ -\frac{n}{r}<s<\frac{n}{r}-1 \, \quad \textit{and} \quad \delta\geq 0.\]
\end{Theorem}
In proof, we see that the parameter $r\in (n,2n)$, $s$ can be as close as to $0$; and hence the convergence occurs in space close to $L^2(\mathbb T^n)$.

\bigskip

\subsection{Hall-MHD}
The Incompressible Hall-MHD equations are given by 
 \begin{align}\label{hmhd}
 \begin{split}
    &u_t+(u\cdot \nabla)u-(b\cdot\nabla)b+\nabla P=\nu \Delta u \\
    &b_t+(u\cdot \nabla)b-(b\cdot \nabla)u+\nabla\times((\nabla\times b)\times b)=\mu\Delta B\\
    &\nabla\cdot u=0, \qquad \nabla\cdot b=0.
    \end{split}
    \end{align}

%For this system, we aim to define determining wavenumbers $\Lambda_{u,r}$ and $\Lambda_{b,r}$ that are optimal. To that end, we define $\Lambda_{u,r}$ identical to Ref.\cite{1} where authors showed that such definition of a determining wavenumber for 3D Naiver Stokes equation is physically optimal. The motivation for this choice is due to the fact that first equation in Hall-MHD is very similar to Navier Stokes equation and has the same scaling. We define $\Lambda_{b,r}$ in analogous way which also physically optimal, as supposed to definitions of wavenumebrs given in Ref.\cite{3} that were not optimal and lacked any comparison with physical phenomenological dissipation scales for Hall-MHD.\\

For this system, we aim to define determining wavenumbers $\Lambda_{u,r}$ and $\Lambda_{b,r}$ that are physically optimal. To this end, we define $\Lambda_{u,r}$ identically to Ref.~\cite{1}, where the authors showed that such a definition of a determining wavenumber for the 3D Navier–Stokes equations is physically optimal. The motivation for this choice stems from the fact that the velocity equation in Hall–MHD closely resembles the Navier–Stokes equations and exhibits the same scaling properties.

We then define $\Lambda_{b,r}$ in an analogous manner, leading again to a physically optimal characterization. This contrasts with the definitions of determining wavenumbers proposed in Ref.~\cite{Han-Liu}, which are not optimal and do not provide any comparison with phenomenological dissipation scales predicted by Hall–MHD turbulence theory.\\

Now we let $0<\gamma\leq 3$ and define $\sigma:=\frac{\gamma-1}{2}$. This implies that $-\frac{1}{2}\leq \sigma\leq 1$. Therefore, we define

\begin{equation}\notag
\Lambda_{u,r}(t):=\min \{ \lambda_q: (L\,\lambda_{p-q})^{\sigma}\, \lambda_{q}^{-1}\,|| u_{p}||_{L^{\infty}}<c_{r}\nu, \, \forall p>q \, \textit{ and } \, \lambda_{q}^{-2}\,|| \nabla u_{\leq q}||_{L^{\infty}}<c_r \nu, \, q\in\mathbb{N} \}
\end{equation}
 \begin{equation}\notag
\Lambda_{b,r}(t):=\min \{ \lambda_q: (L\lambda_{p-q})^{\delta} \, \lambda_{p}^{\frac{n}{r}}\, || b_{p}||_{L^{r}}<c_{r}\mu, \, \forall p>q \, \textit{ and } \,\lambda^{-1}_{q}\, || \nabla b_{\leq q}||_{L^{\infty}}<c_r \mu, \, q\in\mathbb{N} \}
\end{equation} 
where $r\in(n,2n)$ and $\delta\geq 0$. Now  let's denote 
\begin{equation}\label{Hall-Wave}
\begin{split}
\Lambda_U(t):=&\max\{ \Lambda_u(t), \Lambda_v(t)\}\\
\Lambda_B(t):=&\max\{ \Lambda_{b^{(1)}}(t), \Lambda_{b^{(2)}}(t)\}
\end{split}
\end{equation}

% \begin{equation}
 %w(t)\vert_{\leq Q_U(t)}=0, \qquad h(t)\vert_{\leq Q_B(t)}=0.
 %\end{equation}
%Under these considerations, 

Now, we can prove the following result, 

\begin{Theorem}\label{thm-Hall-det}
Let $(u(t),b^{(1)}(t))$ and $(v(t),b^{(2)}(t))$ be two weak solutions of (\ref{hall-mhd}). Let $\Lambda_B(t)$ and $\Lambda_U(t)$ be the wavenumbers defined in (\ref{Hall-Wave}). Let $Q_{U}(t)$ and $Q_{B}(t)$ be such that $\Lambda_U(t)=\Lambda_{Q_{U}(t)}$ and $\Lambda_B(t)=\Lambda_{Q_{B}(t)}$, respectively. Assume that 
\[ \Big( \, (u-v)\big|_{\leq \Lambda_U}, (b^{(1)}-b^{(2)})\big|_{\leq \Lambda_B} \Big) =0. \]
Assume moreover that
\[
\delta>1-\frac{n}{r},
\qquad
\sigma>1-\frac{n}{r}.
\]
Then we have 
\[\lim_{t\to\infty} \big(  \|u-v\|_{H^s} + \|b^{(1)}-b^{(2)}\|_{H^s}\big)=0, \ \ \ \mbox{for} \ \ -\min\left\{\frac{n}{r},\delta,\sigma\right\}<s<\frac{n}{r}-1.\]
\end{Theorem}
Since \(r\in(n,2n)\), the above assumptions imply in particular that \(r>\frac n2\), \(\delta>0\), and, together with \(\sigma=\frac{\gamma-1}{2}\) for \(0<\gamma\le 3\), also \(0<\sigma\le 1\).

\bigskip

\section{Proofs of Theorem \ref{thm-det} and Theorem \ref{thm-Hall-det}}  

\subsection{Proof of Theorem \ref{thm-det}}
We prove Theorem \ref{thm-det} in this section. 
Denote $h(t)=b^{(1)}(t)-b^{(2)}(t)$, which satisfies the equation
\begin{equation}\label{eq-H}
h_t+\nabla\times \left((\nabla\times h )\times b^{(1)} \right)+\nabla\times \left((\nabla\times b^{(2)} )\times h) \right)=\mu\Delta h
\end{equation}
with $\nabla \cdot h=0$. It is clear from the assumption of the theorem that $h|_{\leq Q}\equiv 0$. On the other hand, because of the definition of $\lambda_Q$ we have 
\begin{equation}\label{high-modes}
\begin{split}
 (L\lambda _{q-Q})^{\delta} \lambda_q^{\frac{n}{r}}\|b^{(1), (2)}_q\|_{L^r}<&\ c_r \mu, \ \ \forall \ \ q>Q, \\
\|\nabla b^{(1), (2)}_{\leq Q}\|_{L^\infty}<&\ c_r \mu \lambda_Q.
\end{split}
\end{equation}
In particular, Bernstein's inequality gives
\begin{equation}\label{low-Q-block}
\|b^{(1), (2)}_Q\|_{L^\infty}\lesssim \lambda_Q^{-1}\|\nabla b^{(1), (2)}_{\leq Q}\|_{L^\infty}\lesssim c_r \mu.
\end{equation}

We estimate the $H^{s}$ norm of $h(t)$ for appropriate $s$ in the following. It follows from (\ref{eq-H})
\begin{equation}\label{est-energy1}
\begin{split}
&\frac12\frac{d}{dt} \sum_{q\geq -1} \lambda_q^{2s}\|h_q\|_{L^2(\mathbb T^n)}^2+\mu \sum_{q\geq -1} \lambda_q^{2s+2}\|h_q\|_{L^2(\mathbb T^n)}^2\\
=& -\sum_{q\geq -1} \lambda_q^{2s} \int_{\mathbb T^n} \Delta_q\left((\nabla\times h)\times b^{(1)}\right)\cdot \nabla\times h_q\, d\vec x\\
&-\sum_{q\geq -1} \lambda_q^{2s} \int_{\mathbb T^n} \Delta_q\left((\nabla\times b^{(2)})\times h\right)\cdot \nabla\times h_q\, d\vec x\\
=&: -J-K.
\end{split}
\end{equation}

We start by estimating $J$. First we use Bony's paraproduct to decompose $J$,
\begin{equation}\notag
\begin{split}
J=&\sum_{q\geq -1}\lambda_q^{2s} \int_{\mathbb T^n} \Delta_q\left((\nabla\times h)\times b^{(1)}\right)\cdot \nabla\times h_q\, d\vec x\\
=&\sum_{q\geq -1}\sum_{|p-q|\leq 2} \lambda_q^{2s} \int_{\mathbb T^n} \Delta_q\left((\nabla\times h_{\leq p-2})\times b^{(1)}_p\right)\cdot \nabla\times h_q\, d\vec x\\
&+\sum_{q\geq -1}\sum_{|p-q|\leq 2} \lambda_q^{2s} \int_{\mathbb T^n} \Delta_q\left((\nabla\times h_{p})\times b^{(1)}_{\leq p-2}\right)\cdot \nabla\times h_q\, d\vec x\\
&+\sum_{q\geq -1}\sum_{p\geq q-2} \lambda_q^{2s} \int_{\mathbb T^n} \Delta_q\left((\nabla\times \widetilde{h}_{p})\times b^{(1)}_p\right)\cdot \nabla\times h_q\, d\vec x\\
=&: J_1+J_2+J_3.
\end{split}
\end{equation}
We start by estimating $J_1$. By applying  H\"older's inequality, Bernstein's inequality, we have
\begin{equation}\notag
\begin{split}
|J_1|\leq & \sum_{q>Q} \sum_{\substack {|p-q|\leq 2\\ p>Q+2}}\lambda_q^{2s} \int_{\mathbb T^n} \left|\Delta_q\left((\nabla\times h_{(Q, p-2]})\times b^{(1)}_p\right)\cdot \nabla\times h_q\right|\, d\vec x \\
\leq & \sum_{q>Q}\sum_{|p-q|\leq 2} \lambda_q^{2s} \|\nabla\times h_{(Q, p-2]}\|_{L^{\frac{2r}{r-2}}}\|b^{(1)}_p\|_{L^r} \|\nabla\times h_q\|_{L^2}\\
\lesssim &\ c_r\mu\sum_{q>Q}\sum_{|p-q|\leq 2} \lambda_q^{2s+1} \|h_q\|_{L^2}\lambda_p^{-\frac{n}{r}-\delta} \,\lambda_{Q}^{\delta} \,  \sum_{Q<p'\leq p-2} \lambda_{p'}^{1+\frac{n}{r}}\|h_{p'}\|_{L^2}\\
\lesssim &\ c_r\mu \sum_{q>Q}  \lambda_q^{2s+1-\frac{n}{r}-\delta } \|h_q\|_{L^2} \, \lambda_{Q}^{\delta} \,\sum_{Q<p'\leq q} \lambda_{p'}^{1+\frac{n}{r}}\|h_{p'}\|_{L^2}\\
\lesssim &\ c_r\mu \sum_{q>Q}  \lambda_q^{s+1} \|h_q\|_{L^2} \,  \lambda_{Q-q}^{\delta} \, \sum_{Q<p'\leq q} \lambda_{p'}^{s+1}\|h_{p'}\|_{L^2} \lambda_{q-p'}^{s-\frac{n}{r}}\\
\lesssim &\ c_r\mu \sum_{q\geq -1}  \lambda_q^{2s+2} \|h_q\|_{L^2}^2.
\end{split}
\end{equation}
 where we used $\delta\geq 0$, and for $s-\frac{n}{r}<0$ we applied Jensen's inequality.

To estimate $J_2$, we use the commutator
\[[\Delta_q, b^{(1)}_{\leq p-2}\times \nabla\times ]h_p=\Delta_q\left(b^{(1)}_{\leq p-2}\times(\nabla\times h_p)\right)- b^{(1)}_{\leq p-2}\times \nabla\times \Delta_q(h_p)\]
to further decompose $J_2$,
\begin{equation}\notag
\begin{split}
J_2=&-\sum_{q\geq -1}\sum_{|p-q|\leq 2} \lambda_q^{2s} \int_{\mathbb T^n} \Delta_q\left(b^{(1)}_{\leq p-2}\times (\nabla\times h_{p})\right)\cdot \nabla\times h_q\, d\vec x\\
=&-\sum_{q\geq -1}\sum_{|p-q|\leq 2} \lambda_q^{2s} \int_{\mathbb T^n} [\Delta_q, b^{(1)}_{\leq p-2}\times \nabla\times ]h_p \cdot \nabla\times h_q\, d\vec x\\
&-\sum_{q\geq -1}\sum_{|p-q|\leq 2} \lambda_q^{2s} \int_{\mathbb T^n} b^{(1)}_{\leq q-2}\times \nabla\times \Delta_q(h_p) \cdot \nabla\times h_q\, d\vec x\\
&-\sum_{q\geq -1}\sum_{|p-q|\leq 2} \lambda_q^{2s} \int_{\mathbb T^n} \left(b^{(1)}_{\leq p-2}-b^{(1)}_{\leq q-2}\right)\times \nabla\times \Delta_q(h_p) \cdot \nabla\times h_q\, d\vec x\\
=&:J_{21}+J_{22}+J_{23}.
\end{split}
\end{equation}
Note
\begin{equation}\notag
J_{22}=-\sum_{q\geq -1} \lambda_q^{2s} \int_{\mathbb T^n} b^{(1)}_{\leq q-2}\times (\nabla\times h_q) \cdot \nabla\times h_q\, d\vec x
=0.
\end{equation} 
and, using Lemma \ref{cmmes} with $(r_1,r_2,r_3)=\left(2,\frac{2r}{r-2},r\right)$, we estimate $J_{21}$ and $J_{23}$ according to

\begin{equation}\notag
\begin{split}
|J_{21}|\leq &  \sum_{q>Q} \sum_{\substack {|p-q|\leq 2\\ p>Q+2}}\lambda_q^{2s} \|\nabla\times b^{(1)}_{(Q,p-2]}\|_{L^r} \|h_p\|_{L^{\frac{2r}{r-2}}}\|\nabla\times h_q\|_{L^2}\\
\lesssim &  \sum_{q>Q} \sum_{\substack {|p-q|\leq 2\\ p>Q+2}}\lambda_q^{2s+1} \|h_q\|_{L^2} \lambda_p^{\frac{n}{r}}  \|h_p\|_{L^2}\sum_{Q<p'\leq p-2} \lambda_{p'} \|b^{(1)}_{p'}\|_{L^r}\\
\lesssim &  \sum_{q>Q} \lambda_q^{2s+1+\frac{n}{r}} \|h_q\|^{2}_{L^2} \sum_{Q<p'\leq q} \lambda_{p'} \|b^{(1)}_{p'}\|_{L^r}\\
\lesssim&\ c_r\mu \sum_{q>Q}\lambda_q^{2s+1+\frac{n}{r}} \|h_q\|_{L^2}^2\sum_{Q<p'\leq q} \lambda_{p'}^{1-\frac{n}{r}-\delta}\, \lambda_{Q}^{\delta}\\
\lesssim &\ c_r\mu \sum_{q>Q}\lambda_q^{2s+2} \|h_q\|_{L^2}^2\sum_{Q<p'\leq q} \lambda_{p'-q}^{1-\frac{n}{r}}\, \, \lambda_{Q-p'}^{\delta} \lesssim \ c_r\mu \sum_{q\geq -1}\lambda_q^{2s+2} \|h_q\|_{L^2}^2,
\end{split}
\end{equation}
where we used $\delta\geq 0$, and further we made use of Jensen's inequality for $1-\frac{n}{r}>0$. 
\begin{equation}\notag
\begin{split}
|J_{23}|\lesssim &\sum_{q>Q} \lambda_q^{2s+2}\|h_q\|_{L^2} \|h_q\|_{L^{\frac{2r}{r-2}}} \|b^{(1)}_q\|_{L^r}\\
\lesssim & \sum_{q>Q} \lambda_q^{2s+2+\frac{n}{r}} \|h_q\|^2_{L^2}\|b^{(1)}_q\|_{L^r}\\
\lesssim& \ c_r\mu  \sum_{q>Q} \lambda_q^{2s+2+\frac{n}{r}} \|h_q\|^2_{L^2}\, \lambda_{q}^{-\frac{n}{r}-\delta} \, \lambda_{Q}^{\delta}\\
\lesssim& \ c_r\mu  \sum_{q>Q} \lambda_q^{2s+2} \|h_q\|^2_{L^2}\, \, \lambda_{Q-q}^{\delta}\lesssim \ c_r\mu   \sum_{q\geq -1}\lambda_q^{2s+2}\|h_q\|_{L^2}^2.
\end{split}
\end{equation}
Therefore, $J_2$ is estimated for $\delta\geq 0$, and $n<r$.   

To estimate $J_3$, we split off the cutoff-frequency contribution $p=Q$. Since $h_{\leq Q}=0$, we have $\widetilde h_Q=h_{Q+1}$. Therefore
\begin{equation}\notag
\begin{split}
J_3=&\sum_{Q<q\leq Q+2} \lambda_q^{2s} \int_{\mathbb T^n} \Delta_q\left((\nabla\times \widetilde h_Q)\times b^{(1)}_Q\right)\cdot \nabla\times h_q\, d\vec x\\
&+\sum_{q>Q}\sum_{p\geq \max\{Q+1,q-2\}} \lambda_q^{2s} \int_{\mathbb T^n} \Delta_q\left((\nabla\times \widetilde h_p)\times b^{(1)}_p\right)\cdot \nabla\times h_q\, d\vec x\\
=:&J_{31}+J_{32}.
\end{split}
\end{equation}
\begin{equation}\notag
\begin{split}
|J_{31}|\lesssim& \sum_{Q<q\leq Q+2}\lambda_q^{2s}\|b^{(1)}_Q\|_{L^\infty}\|\nabla\times h_{Q+1}\|_{L^2}\|\nabla\times h_q\|_{L^2}\\
\lesssim&\ c_r\mu\sum_{Q<q\leq Q+2}\lambda_q^{2s}\lambda_{Q+1}\|h_{Q+1}\|_{L^2}\lambda_q\|h_q\|_{L^2}\\
\lesssim&\ c_r\mu\sum_{Q<q\leq Q+2}\lambda_q^{2s+2}\left(\|h_{Q+1}\|_{L^2}^2+\|h_q\|_{L^2}^2\right)\\
\lesssim&\ c_r\mu\sum_{q\geq -1}\lambda_q^{2s+2}\|h_q\|_{L^2}^2,
\end{split}
\end{equation}
where we used (\ref{low-Q-block}) and $\lambda_q\sim \lambda_Q$ for $Q<q\leq Q+2$. For the strictly high-frequency part, the defining property of $\Lambda(t)$ applies to every $p\geq \max\{Q+1,q-2\}$, and thus
\begin{equation}\notag
\begin{split}
|J_{32}|\leq&\sum_{q>Q}\sum_{p\geq \max\{Q+1,q-2\}} \lambda_q^{2s} \left|\int_{\mathbb T^n} \Delta_q\left((\nabla\times \widetilde h_p)\times b^{(1)}_p\right)\cdot \nabla\times h_q\, d\vec x\right|\\
\lesssim&\sum_{q>Q}\sum_{p\geq \max\{Q+1,q-2\}} \lambda_q^{2s} \|\nabla\times \widetilde h_p\|_{L^2}\|b^{(1)}_p\|_{L^r} \|\nabla\times h_q\|_{L^{\frac{2r}{r-2}}}\\
\lesssim&\ c_r\mu\sum_{q>Q}\lambda_q^{2s+1+\frac{n}{r}} \|h_q\|_{L^2}\sum_{p\geq \max\{Q+1,q-2\}}\lambda_p^{1-\frac{n}{r}-\delta}\lambda_Q^\delta \|\widetilde h_p\|_{L^2}\\
\lesssim&\ c_r\mu\sum_{q>Q}\lambda_q^{s+1} \|h_q\|_{L^2}\sum_{p\geq \max\{Q+1,q-2\}}\lambda_p^{s+1}\|\widetilde h_p\|_{L^2}\lambda_{q-p}^{s+\frac{n}{r}}\lambda_{Q-p}^{\delta}\\
\lesssim&\ c_r\mu\sum_{q\geq -1}\lambda_q^{2s+2} \|h_q\|_{L^2}^2,
\end{split}
\end{equation}
since $s+\frac{n}{r}>0$, $\delta\geq 0$, and $\widetilde h_p$ has finite overlap.

Similarly to estimate $K$, we start by decomposing it via the Bony's paraproduct 
\begin{equation}\notag
\begin{split}
K=&\sum_{q\geq -1} \sum_{|p-q|\leq 2} \lambda_q^{2s} \int_{\mathbb T^n} \Delta_q\left((\nabla\times b^{(2)}_{\leq p-2})\times h_p\right)\cdot \nabla\times h_q\, d\vec x\\
&+\sum_{q\geq -1}\sum_{|p-q|\leq 2} \lambda_q^{2s} \int_{\mathbb T^n} \Delta_q\left((\nabla\times b^{(2)}_p)\times h_{\leq p-2}\right)\cdot \nabla\times h_q\, d\vec x\\
&+\sum_{q\geq -1} \sum_{p\geq q-2}\lambda_q^{2s} \int_{\mathbb T^n} \Delta_q\left((\nabla\times \widetilde b^{(2)}_p)\times h_p\right)\cdot \nabla\times h_q\, d\vec x\\
=&: K_1+K_2+K_3. 
\end{split}
\end{equation}
Since $h_{\leq Q}\equiv 0$, we split $K_1$ according to 
\begin{equation}\notag
\begin{split}
K_1=& \sum_{q>Q}\sum_{\substack {|p-q|\leq 2\\ p>Q+2}}\lambda_q^{2s} \int_{\mathbb T^n} \Delta_q\left((\nabla\times b^{(2)}_{(Q, p-2]}\times h_p\right)\cdot \nabla\times h_q\, d\vec x\\
& +\sum_{q>Q}\sum_{\substack {|p-q|\leq 2\\ p>Q+2}}\lambda_q^{2s} \int_{\mathbb T^n} \Delta_q\left((\nabla\times b^{(2)}_{\leq Q}\times h_p\right)\cdot \nabla\times h_q\, d\vec x\\
=&: K_{11}+K_{12}.
\end{split}
\end{equation}
Applying H\"older's inequality, Bernstein's inequality and (\ref{high-modes}) gives
\begin{equation}\notag
\begin{split}
|K_{11}|\lesssim& \sum_{q>Q} \sum_{\substack {|p-q|\leq 2\\ p>Q+2}}\lambda_q^{2s}\|h_p\|_{L^2} \|\nabla\times h_q\|_{L^{\frac{2r}{r-2}}}
\sum_{Q<p'\leq p-2} \|\nabla\times b^{(2)}_{p'}\|_{L^r}\\
\lesssim&\ c_r\mu \sum_{q>Q} \sum_{\substack {|p-q|\leq 2\\ p>Q+2}}\lambda_q^{2s+1+\frac{n}{r}}\|h_p\|_{L^2} \|h_q\|_{L^2}
\sum_{Q<p'\leq p-2}\lambda_{p'}^{1-\frac{n}{r}-\delta} \lambda_{Q}^{\delta} \\
\lesssim &\ c_r\mu \sum_{q>Q}\lambda_q^{2s+1+\frac{n}{r}} \|h_q\|_{L^2}^2\sum_{Q<p'\leq q}\lambda_{p'}^{1-\frac{n}{r}-\delta} \lambda_{Q}^{\delta}\\
\lesssim &\ c_r\mu \sum_{q>Q}\lambda_q^{2s+2} \|h_q\|_{L^2}^2\sum_{Q<p'\leq q}\lambda_{p'-q}^{1-\frac{n}{r}}\,\, \lambda_{p'-Q}^{-\delta}\lesssim \ c_r\mu \sum_{q\geq -1}\lambda_q^{2s+2} \|h_q\|_{L^2}^2
\end{split}
\end{equation}
where we used $1-\frac{n}{r}>0$ and $\delta\geq 0$ in the last step. Similarly we have
\begin{equation}\notag
\begin{split}
|K_{12}|\lesssim& \sum_{q>Q} \sum_{\substack {|p-q|\leq 2\\ p>Q+2}}\lambda_q^{2s} \|\nabla\times b^{(2)}_{\leq Q}\|_{L^\infty} \|h_p\|_{L^2}\|\nabla\times h_q\|_{L^2}\\
\lesssim&\ c_r\mu \sum_{q>Q} \sum_{\substack {|p-q|\leq 2\\ p>Q+2}}\lambda_q^{2s} \lambda_Q \|h_p\|_{L^2}\|\nabla\times h_q\|_{L^2}\\
\lesssim&\ c_r\mu \sum_{q>Q} \sum_{\substack {|p-q|\leq 2\\ p>Q+2}}\lambda_q^{2s+2}\, \lambda_{q-Q}^{-1}\, \|h_p\|_{L^2}\| h_q\|_{L^2}\lesssim \ c_r\mu \sum_{q\geq -1}\lambda_q^{2s+2} \|h_q\|_{L^2}^2.
\end{split}
\end{equation}
where we used the fact that $\lambda_{q-Q}^{-1}<1$ for $q>Q$.
The term $K_2$ is handled in a similar way,
\begin{equation}\notag
\begin{split}
|K_2|\leq & \sum_{q>Q} \sum_{\substack {|p-q|\leq 2\\ p>Q+2}}\lambda_q^{2s} \int_{\mathbb T^n} \left|\Delta_q\left((\nabla\times b^{(2)}_p)\times h_{(Q, p-2]}\right)\cdot \nabla\times h_q\right|\, d\vec x\\
\leq & \sum_{q>Q} \sum_{\substack {|p-q|\leq 2\\ p>Q+2}}\lambda_q^{2s}\|\nabla\times b^{(2)}_p\|_{L^r}\|h_{(Q,p-2]}\|_{L^{\frac{2r}{r-2}}}\|\nabla\times h_q\|_{L^2}\\
%\leq & \sum_{q>Q} \sum_{\substack {|p-q|\leq 2\\ p>Q+2}}\lambda_q^{2s+1}\, \lambda_{p}\, \|  b^{(2)}_p\|_{L^r}\, \| h_q\|_{L^2} \,\|h_{(Q,p-2]}\|_{L^{\frac{2r}{r-2}}} \\
\leq &\ c_r\mu  \sum_{q>Q} \sum_{\substack {|p-q|\leq 2\\ p>Q+2}}\lambda_q^{2s+1}\lambda_p^{1-\frac{n}{r}-\delta} \, \lambda_{Q}^{\delta} \, \|h_q\|_{L^2}\|h_{(Q,p-2]}\|_{L^{\frac{2r}{r-2}}}\\
\leq &\ c_r\mu  \sum_{q>Q}\lambda_q^{2s+2-\frac{n}{r} -\delta} \, \lambda_{Q}^{\delta}\,\|h_q\|_{L^2}\sum_{Q<p'\leq q}\lambda_{p'}^{\frac{n}{r}}\|h_{p'}\|_{L^2}\\
\leq &\ c_r\mu  \sum_{q>Q}\lambda_q^{s+1}\|h_q\|_{L^2}\,  \lambda_{Q-q}^{\delta} \, \sum_{Q<p'\leq q}\lambda_{p'}^{s+1}\|h_{p'}\|_{L^2} \lambda_{q-p'}^{s+1-\frac{n}{r}},
\end{split}
\end{equation}
where we use the fact that $\lambda_{Q-q}^{\delta} \leq 1$ for $\delta\geq 0$. In addition, using Jensen's inequality for $s+1-\frac{n}{r}<0$ yields
\begin{equation}\notag
\begin{split}
|K_2| \leq& c_r\mu  \sum_{q>Q}\lambda_q^{2s+2}\|h_q\|_{L^2}^2+c_r\mu \sum_{q>Q} \sum_{Q<p'\leq q}\lambda_{p'}^{2s+2}\|h_{p'}\|_{L^2}^2 \lambda_{q-p'}^{s+1-\frac{n}{r}}\\
\leq& c_r\mu  \sum_{q\geq -1}\lambda_q^{2s+2}\|h_q\|_{L^2}^2.
\end{split}
\end{equation}
At the end, we also split $K_3$ into the cutoff-frequency contribution $p=Q+1$ and the strictly high-frequency remainder:
\begin{equation}\notag
\begin{split}
K_3=&\sum_{Q<q\leq Q+3}\lambda_q^{2s} \int_{\mathbb T^n} \Delta_q\left((\nabla\times \widetilde b^{(2)}_{Q+1})\times h_{Q+1}\right)\cdot \nabla\times h_q\, d\vec x\\
&+\sum_{q>Q}\sum_{p\geq \max\{Q+2,q-2\}}\lambda_q^{2s} \int_{\mathbb T^n} \Delta_q\left((\nabla\times \widetilde b^{(2)}_p)\times h_p\right)\cdot \nabla\times h_q\, d\vec x\\
=:&K_{31}+K_{32}.
\end{split}
\end{equation}
For $K_{31}$, using (\ref{high-modes}) and Bernstein's inequality we obtain
\begin{equation}\notag
\begin{split}
\|\nabla\times \widetilde b^{(2)}_{Q+1}\|_{L^\infty}
\lesssim&\ \|\nabla b^{(2)}_{\leq Q}\|_{L^\infty}+\lambda_{Q+1}^{1+\frac{n}{r}}\|b^{(2)}_{Q+1}\|_{L^r}+\lambda_{Q+2}^{1+\frac{n}{r}}\|b^{(2)}_{Q+2}\|_{L^r}\\
\lesssim&\ c_r\mu \lambda_Q.
\end{split}
\end{equation}
Hence
\begin{equation}\notag
\begin{split}
|K_{31}|\lesssim& \sum_{Q<q\leq Q+3}\lambda_q^{2s}\|\nabla\times \widetilde b^{(2)}_{Q+1}\|_{L^\infty}\|h_{Q+1}\|_{L^2}\|\nabla\times h_q\|_{L^2}\\
\lesssim&\ c_r\mu\sum_{Q<q\leq Q+3}\lambda_q^{2s}\lambda_Q\|h_{Q+1}\|_{L^2}\lambda_q\|h_q\|_{L^2}\\
\lesssim&\ c_r\mu\sum_{Q<q\leq Q+3}\lambda_q^{2s+2}\left(\|h_{Q+1}\|_{L^2}^2+\|h_q\|_{L^2}^2\right)\\
\lesssim&\ c_r\mu\sum_{q\geq -1}\lambda_q^{2s+2}\|h_q\|_{L^2}^2.
\end{split}
\end{equation}
For $K_{32}$, note that for every $p\geq Q+2$ all dyadic blocks entering $\widetilde b^{(2)}_p$ have index strictly larger than $Q$, so the determining-wavenumber bound applies to each of them. Therefore
\begin{equation}\notag
\begin{split}
|K_{32}|\lesssim&\sum_{q>Q}\sum_{p\geq \max\{Q+2,q-2\}}\lambda_q^{2s}\|\nabla\times \widetilde b^{(2)}_p\|_{L^r}\|h_p\|_{L^2}\|\nabla\times h_q\|_{L^{\frac{2r}{r-2}}}\\
\lesssim&\ c_r\mu\sum_{q>Q}\lambda_q^{2s+1+\frac{n}{r}}\|h_q\|_{L^2}\sum_{p\geq \max\{Q+2,q-2\}}\lambda_p^{1-\frac{n}{r}-\delta}\lambda_Q^\delta\|h_p\|_{L^2}\\
\lesssim&\ c_r\mu\sum_{q>Q}\lambda_q^{s+1}\|h_q\|_{L^2}\sum_{p\geq \max\{Q+2,q-2\}}\lambda_p^{s+1}\|h_p\|_{L^2}\lambda_{q-p}^{s+\frac{n}{r}}\lambda_{Q-p}^{\delta}\\
\lesssim&\ c_r\mu\sum_{q\geq -1}\lambda_q^{2s+2}\|h_q\|_{L^2}^2,
\end{split}
\end{equation}
since $s+\frac{n}{r}>0$ and $\delta\geq 0$. 

Taken together the estimates of $J$ and $K$ with (\ref{est-energy1}), we conclude that for a small enough constant $c_r$
\begin{equation}\notag
\begin{split}
\frac{d}{dt} \sum_{q\geq -1} \lambda_q^{2s}\|h_q\|_{L^2(\mathbb T^n)}^2+\mu \sum_{q\geq -1} \lambda_q^{2s+2}\|h_q\|_{L^2(\mathbb T^n)}^2\leq 0.
\end{split}
\end{equation}
which implies that
\[\lim_{t\to\infty} \|h(t)\|_{H^s}=0, \ \ \ \mbox{for} \ \ -\frac{n}{r}<s<\frac{n}{r}-1 \, \, \textit{ and }\, \quad \delta\geq 0.\]
It completes the proof of Theorem \ref{thm-det}.

\bigskip

\subsection{Proof of Theorem \ref{thm-Hall-det}}

To prove Theorem \ref{thm-Hall-det}, let
\[
w(t):=u(t)-v(t),
\qquad
h(t):=b^{(1)}(t)-b^{(2)}(t).
\]
Then $(w,h)$ solves
\begin{equation}\label{Hall-diff-w}
w_t+(w\cdot\nabla)u +(v\cdot\nabla)w-(h\cdot\nabla)b^{(1)}-(b^{(2)}\cdot\nabla)h+\nabla P=\nu\Delta w,
\end{equation}
\begin{equation}\label{Hall-diff-h}
\begin{split}
h_t+(u\cdot\nabla)h +(w\cdot\nabla)b^{(2)}-(h\cdot\nabla)u-(b^{(2)}\cdot\nabla)w \hspace{1.8cm}\\
+\nabla\times((\nabla\times h)\times b^{(1)})
+\nabla\times((\nabla\times b^{(2)})\times h)=\mu\Delta h,
\end{split}
\end{equation}
with $\nabla\cdot w=\nabla\cdot h=0$. The determining-mode hypothesis gives
\begin{equation}\label{Hall-low-vanishing}
w(t)\vert_{\leq Q_U(t)}=0,
\qquad
h(t)\vert_{\leq Q_B(t)}=0.
\end{equation}
Moreover, by the definitions of $\Lambda_U$ and $\Lambda_B$,
\begin{equation}\label{Hall-high-modes}
\begin{split}
(L\lambda_{q-Q_U})^\sigma \lambda_q^{-1}\|u_q\|_{L^\infty}+
(L\lambda_{q-Q_U})^\sigma \lambda_q^{-1}\|v_q\|_{L^\infty}
&\lesssim c_r\nu,\qquad q>Q_U,\\
(L\lambda_{q-Q_B})^\delta \lambda_q^{\frac{n}{r}}\|b_q^{(1)}\|_{L^r}+
(L\lambda_{q-Q_B})^\delta \lambda_q^{\frac{n}{r}}\|b_q^{(2)}\|_{L^r}
&\lesssim c_r\mu,\qquad q>Q_B,
\end{split}
\end{equation}
and
\begin{equation}\label{Hall-low-modes}
\|\nabla u_{\leq Q_U}\|_{L^\infty}+\|\nabla v_{\leq Q_U}\|_{L^\infty}
\lesssim c_r\nu \lambda_{Q_U}^2,
\qquad
\|\nabla b^{(1)}_{\leq Q_B}\|_{L^\infty}+\|\nabla b^{(2)}_{\leq Q_B}\|_{L^\infty}
\lesssim c_r\mu \lambda_{Q_B}.
\end{equation}

Applying $\Delta_q$ to (\ref{Hall-diff-w}) and (\ref{Hall-diff-h}), testing with
$\lambda_q^{2s}w_q$ and $\lambda_q^{2s}h_q$, integrating over $\mathbb{T}^n$, and summing in
$q\geq -1$, we obtain
\begin{equation}\notag
\begin{split}
\frac{1}{2}\, \frac{d}{dt}\, \sum_{q\geq -1}&{\lambda_{q}^{2s} || w_{q}||^{2}_{L^{2}}}
+ \frac{1}{2}\, \frac{d}{dt}\, \sum_{q\geq -1}{\lambda_{q}^{2s} || h_{q}||^{2}_{L^{2}}}
+\nu\, \sum_{q\geq -1}{\lambda_{q}^{2s+2} || w_{q}||^{2}_{L^{2}}}
+\mu\,\sum_{q\geq -1}{\lambda_{q}^{2s+2} || h_{q}||^{2}_{L^{2}}}\\ \notag
&=-\sum_{q\geq -1}{\lambda^{2s}_{q}} \int_{\mathbb{T}^n}{\Delta_{q}((w\cdot \nabla)u)\cdot w_{q}dx}
-\sum_{q\geq -1}{\lambda^{2s}_{q}} \int_{\mathbb{T}^n}{\Delta_{q}((v\cdot \nabla)w)\cdot w_{q}dx}\\ \notag
&+\sum_{q\geq -1}{\lambda^{2s}_{q}} \int_{\mathbb{T}^n}{\Delta_{q}((h\cdot \nabla)b^{(1)})\cdot w_{q}dx}
+\sum_{q\geq -1}{\lambda^{2s}_{q}} \int_{\mathbb{T}^n}{\Delta_{q}((b^{(2)}\cdot \nabla)h)\cdot w_{q}dx}\\ \notag
&-\sum_{q\geq -1}{\lambda^{2s}_{q}} \int_{\mathbb{T}^n}{\Delta_{q}((u\cdot \nabla)h)\cdot h_{q}dx}
-\sum_{q\geq -1}{\lambda^{2s}_{q}} \int_{\mathbb{T}^n}{\Delta_{q}((w\cdot \nabla)b^{(2)})\cdot h_{q}dx}\\ \notag
&+\sum_{q\geq -1}{\lambda^{2s}_{q}} \int_{\mathbb{T}^n}{\Delta_{q}((h\cdot \nabla)u)\cdot h_{q}dx}
+\sum_{q\geq -1}{\lambda^{2s}_{q}} \int_{\mathbb{T}^n}{\Delta_{q}((b^{(2)}\cdot \nabla)w)\cdot h_{q}dx}\\ \notag
&-\sum_{q\geq -1}{\lambda^{2s}_{q}} \int_{\mathbb{T}^n}{\Delta_q\big(\nabla\times((\nabla\times h)\times b^{(1)})\big)\cdot h_q\,dx}\\ \notag
&-\sum_{q\geq -1}{\lambda^{2s}_{q}} \int_{\mathbb{T}^n}{\Delta_q\big(\nabla\times((\nabla\times b^{(2)})\times h)\big)\cdot h_q\,dx}\\
&=:I_1+I_2+I_3+I_4+I_5+I_6+I_7+I_8+J+K.
\end{split}
\end{equation}

The terms $I_1$ and $I_2$ are estimated exactly as in Ref. \cite{1}, while the Hall terms
$J$ and $K$ are the same magnetic terms already treated in Subsection 4.1 with $Q=Q_B$.
Therefore it remains to estimate $I_3,\dots,I_8$.

\subsubsection{Estimate for $I_3$:}

Using Bony's paraproduct we can write $I_3$ according to 

\begin{equation}
\begin{split}
&\quad I_{3}= \sum_{q\geq -1}{\lambda^{2s}_{q}} \int_{\mathbb{T}^n}{\Delta_{q}((h\cdot \nabla)b^{(1)})\cdot w_{q}dx}\\
&=\sum_{q\geq -1}\sum_{\vert p-q\vert\leq 2}{\lambda^{2s}_{q}} \int_{\mathbb{T}^n}{\Delta_{q}( h_{\leq p-2}\cdot \nabla b^{(1)}_p) \cdot w_{q}dx}\\ \notag
&+ \sum_{q\geq -1}\sum_{\vert p-q\vert\leq 2}{\lambda^{2s}_{q}} \int_{\mathbb{T}^n}{\Delta_{q}( h_{p}\cdot \nabla b^{(1)}_{\leq p-2}) \cdot w_{q}dx}\\ \notag
&+ \sum_{q\geq -1}\sum_{ p\geq q-2}{\lambda^{2s}_{q}} \int_{\mathbb{T}^n}{\Delta_{q}( \tilde{h}_{p}\cdot \nabla b^{(1)}_p) \cdot w_{q}dx}=I_{31}+I_{32}+I_{33} 
\end{split}
\end{equation}

Note that based on our assumption, we have $h\vert_{\leq Q_B}=0$. We note that $I_{31}$ only consists of higher modes and hence 
\begin{equation}\notag
\begin{split}
|I_{31}|\leq & \sum_{q\geq -1}\sum_{\vert p-q\vert\leq 2}{\lambda^{2s}_{q}}{\left|  \int_{\mathbb{T}^n}{ \Delta_{q}( h_{\leq p-2}\cdot \nabla b^{(1)}_p) \cdot w_{q}dx} \right| }\\
&\leq \sum_{q>Q_B}\sum_{\substack{\vert p-q\vert\leq 2\\ p>Q_{B}+2}}{\lambda^{2s}_{q} \, \lambda_{q} \, ||b_{p}^{(1)}||_{r}\, ||w_q||_{\frac{2r}{r-2}}} \sum_{Q_B< p^{\prime}\leq p-2}{||h_{p^{\prime}}||_2}  
\end{split}
\end{equation}
Using definition of $\Lambda_B$, it follows that
\begin{equation}\notag
\begin{split}
|I_{31}| &\lesssim c_r \mu \sum_{q>Q_B}\sum_{\substack{\vert p-q\vert\leq 2\\ p>Q_B+2}}{\lambda^{2s+1}_{q} \, \lambda^{-\delta-\frac{n}{r}}_{p} \, \Lambda_{B}^{\delta} \, \lambda_{q}^{\frac{n}{r}}\,||w_q||_{2}} \sum_{Q_B< p^{\prime}\leq p-2}{||h_{p^{\prime}}||_2} \\ 
&\lesssim c_r \mu \sum_{q>Q_B}{\lambda^{2s+1-\delta}_{q} \, \lambda^{\delta}_{Q_B} \,||w_q||_{2}} \sum_{Q_B< p^{\prime}\leq q}{||h_{p^{\prime}}||_2}\\ 
%&\lesssim c_r \mu \sum_{q>Q_h}{\lambda^{s+1}_{q} \,||w_q||_{2}} \sum_{Q_h< p^{\prime}\leq q}{ \lambda_{p^{\prime}}^{s+1} \, ||h_{p^{\prime}}||_2 \, \lambda_{q}^{s-\delta}\, \lambda^{\delta}_{Q_{h}} \,\lambda_{p^{\prime}}^{-s-1}\, }\\ 
&\lesssim c_r \mu \sum_{q>Q_B}{\lambda^{s+1}_{q} \,||w_q||_{2}} \sum_{Q_B< p^{\prime}\leq q}{ \lambda_{p^{\prime}}^{s+1}\, ||h_{p^{\prime}}||_2 \, \lambda_{q-p^{\prime}}^{s-\delta}\, \lambda_{Q_B-p^{\prime}}^{\delta}\, \lambda_{p^{\prime}}^{-1} }\\ 
&\lesssim c_r \mu \sum_{q>Q_B}{\lambda^{s+1}_{q} \,||w_q||_{2}} \sum_{Q_B< p^{\prime}\leq q}{ \lambda_{p^{\prime}}^{s+1}\, ||h_{p^{\prime}}||_2 \, \lambda_{q-p^{\prime}}^{s-\delta} }
\end{split}
\end{equation}
where in last step we assumed $\delta\geq 0$. For $s-\delta<0$, and we further use Young and Jensen's inequalities,

\begin{equation}\notag
\begin{split}
|I_{31}|&\lesssim c_r \mu \sum_{q>Q_B}{\lambda^{2s+2}_{q} \,||w_q||^{2}_{2}} + c_r \mu\,  \sum_{q>Q_B} \big( \sum_{Q_B< p^{\prime}\leq q}{ \lambda_{p^{\prime}}^{s+1}\, ||h_{p^{\prime}}||_2 \, \lambda_{q-p^{\prime}}^{s-\delta}} \big)^2\\ \notag
&\lesssim c_r \mu\, \sum_{q>Q_B}{\lambda^{2s+2}_{q} \,||w_q||^{2}_{2}} + c_r \mu\, \sum_{q>Q_B}{\lambda^{2s+2}_{q} \,||h_q||^{2}_{2}}\\ \notag
&\lesssim c_r \mu\, \sum_{q\geq -1}{\lambda^{2s+2}_{q} \,||w_q||^{2}_{2}} + c_r \mu\, \sum_{q\geq -1}{\lambda^{2s+2}_{q} \,||h_q||^{2}_{2}}
\end{split}
\end{equation}
To estimate $I_{32}$, observe that $h\vert_{\leq Q_B}=0$. Furthermore, by utilizing the wavenumber $Q_B$, we can split $I_{32}$ into two terms, and using the fact that $(Q_B,p-2]$ is empty for $p-2\leq Q_B$, it follows that

\begin{equation}\notag
\begin{split}
   I_{32}&=\sum_{q\geq -1}\sum_{\vert p-q\vert\leq 2}{\lambda^{2s}_{q}} \int_{\mathbb{T}^n}{\Delta_{q}( h_{p}\cdot \nabla b^{(1)}_{\leq p-2}) \cdot w_{q}dx}\\ 
    &=\sum_{p>Q_B}\sum_{\vert p-q\vert\leq 2}{\lambda^{2s}_{q}} \int_{\mathbb{T}^n}{\Delta_{q}( h_{p}\cdot \nabla b^{(1)}_{(Q_B,p-2] }) \cdot w_{q}dx}\\ 
    &+\sum_{p>Q_B}\sum_{\vert p-q\vert\leq 2}{\lambda^{2s}_{q}} \int_{\mathbb{T}^n}{\Delta_{q}( h_{p}\cdot \nabla b^{(1)}_{\leq Q_B}) \cdot w_{q}dx}=I_{321}+I_{322} 
\end{split}
\end{equation}
Now we first estimate $I_{321}$ according to 

\begin{equation}\notag
\begin{split}
   |I_{321}| &\leq  \sum_{p>Q_B}\sum_{\substack{\vert p-q\vert\leq 2\\ q>Q_B+2}}{\lambda^{2s}_{q}\, ||w_p||_{2} \, ||h_q||_{\frac{2r}{r-2}} } \sum_{Q_B<p^{\prime}\leq p-2}{|| \nabla b^{(1)}_{p^{\prime}}||_{r}} \\
      &\lesssim  \sum_{p>Q_B}{\lambda^{2s}_{p}\, ||w_p||_{2} \, \lambda_{p}^{\frac{n}{r}}\, ||h_p||_{2} } \sum_{Q_B<p^{\prime}\leq p-2}{ \lambda_{p^{\prime}}\, || b^{(1)}_{p^{\prime}}||_{r}}\\
      &\lesssim c_r\, \mu  \sum_{p>Q_B} {\lambda^{2s}_{p}\, ||w_p||_{2} \, \lambda_{p}^{\frac{n}{r}}\, ||h_p||_{2} } \sum_{Q_B<p^{\prime}\leq p-2}{ \lambda_{p^{\prime}}\cdot \lambda_{p^{\prime}}^{-\delta-\frac{n}{r}}\, \Lambda_B^{\delta}} \\ 
     % &\lesssim c_r\, \mu  \sum_{q>Q_h}{\lambda^{2s+2}_{q}\, ||w_q||_{2} \,||h_q||_{2} } \sum_{Q_h<p^{\prime}\leq q}{ \lambda_{p^{\prime}}^{1-\delta-\frac{n}{r}}\, \lambda_{q}^{-2+\frac{n}{r}} \, \lambda_{Q_h}^{\delta} } \\
      &\lesssim c_r\, \mu  \sum_{p>Q_B} {\lambda^{s+1}_{p}\, ||w_p||_{2} \, \lambda^{s+1}_{p}\,||h_p||_{2} } \sum_{Q_B<p^{\prime}\leq p-2}{ \big( \lambda_{p-p^{\prime}}\big)^{-2+\frac{n}{r}}\, \lambda_{Q_B-p^{\prime}}^{\delta} \, \lambda_{p^{\prime}}^{-1} } 
\end{split}
\end{equation}
Next, we consider $-2+\frac{n}{r}<0$ which in turn implies $r>\frac{n}{2}$, and since $\delta\geq 0$ by our previous assumption, it follows that 

\begin{equation}\notag
\begin{split}
   |I_{321}|   &%\lesssim c_r\, \mu  {\lambda^{s+1}_{q}\, ||w_q||_{2} \, \lambda^{s+1}_{q}\,||h_q||_{2} } \\ 
   \lesssim c_r \mu\, \sum_{p>Q_B}{\lambda^{2s+2}_{q} \,||w_p||^{2}_{2}} + c_r \mu\, \sum_{p>Q_B}{\lambda^{2s+2}_{p} \,||h_p||^{2}_{2}}\\
&\lesssim c_r \mu\, \sum_{p\geq -1}{\lambda^{2s+2}_{p} \,||w_p||^{2}_{2}} + c_r \mu\, \sum_{p\geq -1}{\lambda^{2s+2}_{p} \,||h_p||^{2}_{2}}
\end{split}
\end{equation}

Now we estimate $I_{322}$ according to
\begin{equation}\notag
\begin{split}
   |I_{322}|&\leq \sum_{p>Q_B}\sum_{\substack{\vert p-q\vert\leq 2\\ q>Q_B+2}}{\lambda^{2s}_{q}} \int_{\mathbb{T}^n}{ | \Delta_{q}( h_{p}\cdot \nabla b^{(1)}_{\leq Q_B}) \cdot w_{q} | dx} \\
   &\lesssim  \sum_{p>Q_B}\sum_{\substack{\vert p-q\vert\leq 2\\ q>Q_B+2}} {\lambda^{2s}_{q} \, ||h_p||_{2}\,||w_q||_{2}\, ||\nabla b^{(1)}_{\leq Q_B}||_{\infty}}\\ 
   &\lesssim c_r \mu \sum_{p>Q_B}\sum_{\substack{\vert p-q\vert\leq 2\\ q>Q_B+2}} {\lambda^{2s}_{q} \, \Lambda_B\, ||h_p||_{2}\,||w_q||_{2}\, } \\
   &\lesssim c_r \mu \sum_{p>Q_B} {\lambda^{2s+2}_{p}\, ||h_p||_{2}\,||w_p||_{2} \, \lambda_{p}^{-2}\, \lambda_{Q_B}}
\end{split}
\end{equation}
where we used the definition of wavenumber. Since for $p>Q_B$, $\lambda_{p}^{-2}\, \lambda_{Q_B}<1$, we have that
\begin{equation} \notag
\begin{split}
   |I_{322}|&\leq  %c_r \mu \sum_{q>Q_h}{\lambda^{2s+2}_{q}\, ||h_q||_{2}\,||w_q||_{2} } \hspace{2.6cm}\\ \notag
   \lesssim c_r \mu \sum_{p>Q_B} {\lambda^{2s+2}_{p}\, ||w_p||_{2}^{2} } +  c_r \mu \sum_{p>Q_B} {\lambda^{2s+2}_{p}\, ||h_p||_{2}^{2} } \\ \notag
    &\lesssim c_r \mu \sum_{p\geq -1} {\lambda^{2s+2}_{p}\, ||w_p||_{2}^{2} } +  c_r \mu \sum_{p\geq -1} {\lambda^{2s+2}_{p}\, ||h_p||_{2}^{2} } 
\end{split}
\end{equation}
Finally we focus on the term $I_{33}$: 

\begin{equation} \notag
\begin{split}
    I_{33}&= \sum_{q\geq -1}\sum_{ p\geq q-2}{\lambda^{2s}_{q}} \int_{\mathbb{T}^n}{\Delta_{q}( \tilde{h}_{p}\cdot \nabla b^{(1)}_p) \cdot w_{q}dx}\\ \notag
    &= \sum_{p\leq Q_B}\sum_{ q\leq p+2}{\lambda^{2s}_{q}} \int_{\mathbb{T}^n}{\Delta_{q}( \tilde{h}_{p}\cdot \nabla b^{(1)}_p) \cdot w_{q}dx} \\
&\quad    +\sum_{p>Q_B}\sum_{ q\leq p+2}{\lambda^{2s}_{q}} \int_{\mathbb{T}^n}{\Delta_{q}( \tilde{h}_{p}\cdot \nabla b^{(1)}_p) \cdot w_{q}dx}\\
&=I_{331}+I_{332}
    \end{split}
\end{equation}
We note that we have only few low frequency terms, i.e, 

\begin{align}\notag
  I_{331}= \sum_{p\leq Q_B}\sum_{ q\leq p+2}{\lambda^{2s}_{q}} \int_{\mathbb{T}^n}{\Delta_{q}( \tilde{h}_{p}\cdot \nabla b^{(1)}_p) \cdot w_{q}dx} =  \sum_{ q\leq Q_B+2}{\lambda^{2s}_{q}} \int_{\mathbb{T}^n}{\Delta_{q}( h_{Q_B+1}\cdot \nabla b^{(1)}_{Q_B}) \cdot w_{q}dx}
\end{align}
This is due to the fact that $h\vert_{\leq Q_B}=0$, and $\tilde{h}_p=h_{p-1}+h_{p}+h_{p+1}$. Now estimating it yields 
\begin{equation}\notag
\begin{split}
  | I_{331} | & %\sum_{ q\leq Q+2}{\lambda^{2s}_{q}} \int_{\mathbb{T}^n}{| \Delta_{q}( h_{Q+1}\cdot \nabla b^{(1)}_Q) \cdot w_{q} | dx} 
   \leq \sum_{ -1\leq q\leq Q_B+2}{{\lambda^{2s}_{q}} ||\nabla b^{(1)}_{Q_B}||_{\infty}\, || h_{Q_B+1}||_{2}\, || w_{q}||_{2}}\\
   &\leq \sum_{ -1\leq q\leq Q_B+2}{{\lambda^{2s}_{q}} \,  ||\nabla b^{(1)}_{\leq Q_B}||_{\infty}\, || h_{Q_B+1}||_{2}\, || w_{q}||_{2}}\\ 
 &\lesssim c_r \mu \, \sum_{  q\geq -1}{\lambda^{s+1}_{q} || h_{q}||_{2}\, \lambda^{s+1}_{q} || w_{q}||_{2} \,\, \lambda_{Q_B}\, \lambda_{q}^{-2}} 
 \end{split}
\end{equation}
where $\lambda_{Q_B}\, \lambda_{q}^{-2}<1$ for $q\geq Q_B$. Thus, 
\begin{equation}\notag
\begin{split}
  | I_{331} | &\lesssim c_r \mu \, \sum_{  q\geq -1}{\lambda^{2s+2}_{q} || h_{q}||_{2}\, || w_{q}||_{2} } \\
  &\lesssim  c_r \mu \, \sum_{  q\geq -1}{\lambda^{2s+2}_{q} || h_{q}||^{2}_{2} } + c_r \mu \, \sum_{  q\geq -1}{ \lambda^{2s+2}_{q} || w_{q}||^{2}_{2} }.
\end{split}
\end{equation}
Finally we estimate $I_{332}$ according to

\begin{equation}\notag
\begin{split}
  | I_{332 }|&= \sum_{p> Q_B}\sum_{ q\leq p+2}{\lambda^{2s}_{q}} \int_{\mathbb{T}^n}{| \Delta_{q}( \tilde{h}_{p}\cdot \nabla b^{(1)}_p) \cdot w_{q} | dx} \\
  &\leq \sum_{p> Q_B} { ||\nabla b^{(1)}_p||_{r}\, || \tilde{h}_{p}||_{\frac{2r}{r-2}} }  \sum_{  q\leq p+2}{\lambda^{2s}_{q} || w_{q}||_{2}}   \\ 
 &\lesssim \sum_{  p>Q_B}{ \, \lambda_{p} || b^{(1)}_{p}||_{r}\, \lambda_{p}^{\frac{n}{r}}\, || \tilde{h}_{p}||_{2} }  \sum_{  q\leq p+2}{\lambda^{2s}_{q} || w_{q}||_{2}}\\
 & \lesssim c_r \mu \, \sum_{  p>Q_B}{ \lambda_{p}^{s+1} \, || \tilde{h}_{p}||_{2}\, \lambda_{p}^{-s-\delta}\, \Lambda^{\delta}_B }  \sum_{  q\leq p+2}{\lambda^{2s}_{q} || w_{q}||_{2}}\\ 
 &\lesssim c_r \mu \, \sum_{  p>Q_B}{ \lambda_{p}^{s+1} \, || \tilde{h}_{p}||_{2} }  \sum_{  q\leq p+2}{\lambda^{s+1}_{q} || w_{q}||_{2} \, \lambda_{q-p}^{s+\delta} \, \lambda_{q-Q_B}^{-\delta-1}\, \lambda^{-1}_{Q_B}}
\end{split}
\end{equation}
Assuming $\delta>-1$ observe that $\lambda_{q-Q_B}^{-\delta-1}<1$ for $q>Q_B$, and for $s>-\delta$ using Jensen’s inequality yields
\begin{equation}\notag
\begin{split}
| I_{332} | & \lesssim c_r \mu \, \sum_{  p>Q_B}{ \lambda_{p}^{s+1} \, || \tilde{h}_{p}||_{2} }  \sum_{  q\leq p+2}{\lambda^{s+1}_{q} || w_{q}||_{2} \, \lambda_{q-p}^{s+\delta} } \\ 
 &\lesssim c_r \mu \, \sum_{  p>Q_B}{ \lambda_{p}^{2s+2} \, || \tilde{h}_{p}||^{2}_{2} } + c_r \mu \, \sum_{  p>Q_B} \Big( \sum_{  q\leq p+2}{\lambda^{s+1}_{q} || w_{q}||_{2} \, \lambda_{q-p}^{s+\delta} } \Big)^2\\
&\lesssim  c_r \mu \, \sum_{  p>Q_B}{ \lambda_{p}^{2s+2} \, || \tilde{h}_{p}||^{2}_{2} } + c_r \mu \, \sum_{  q\geq -1} {\lambda^{2s+2}_{q} || w_{q}||^{2}_{2}  }\\ 
 &\lesssim c_r \mu \, \sum_{  p\geq -1}{ \lambda_{p}^{2s+2} \, || \tilde{h}_{p}||^{2}_{2} } + c_r \mu \, \sum_{  q\geq -1} {\lambda^{2s+2}_{q} || w_{q}||^{2}_{2}  }
 \end{split}
\end{equation}
By combining all estimates above, we conclude that
\begin{equation}\label{I_3}
 | I_{3} |\lesssim    c_r \mu \, \sum_{  q\geq -1}{ \lambda_{q}^{2s+2} \, || h_{q}||^{2}_{2} } + c_r \mu \, \sum_{  q\geq -1} {\lambda^{2s+2}_{q} || w_{q}||^{2}_{2}  } 
\end{equation}

\vspace{0.2cm}

\subsubsection{Estimate for $I_{4}$}

We start estimating $I_{4}$. First, we use Bony's paraproduct to write $I_4$ as
\begin{equation} \notag
\begin{split}
I_{4}&= \sum_{q\geq -1}{\lambda^{2s}_{q}} \int_{\mathbb{T}^n}{\Delta_{q}(b^{(2)}\cdot \nabla h )\cdot w_{q}dx}\\
&=\sum_{q\geq -1}\sum_{\vert p-q\vert\leq 2}{\lambda^{2s}_{q}} \int_{\mathbb{T}^n}{\Delta_{q}( b^{(2)}_{\leq p-2}\cdot \nabla h_{p}) \cdot w_{q}dx}\\
&\quad+ \sum_{q\geq -1}\sum_{\vert p-q\vert\leq 2}{\lambda^{2s}_{q}} \int_{\mathbb{T}^n}{\Delta_{q}( b^{(2)}_{p}\cdot \nabla h_{\leq p-2}) \cdot w_{q}dx}\\ 
&\quad+ \sum_{q\geq -1}\sum_{ p\geq q-2}{\lambda^{2s}_{q}} \int_{\mathbb{T}^n}{\Delta_{q}( \tilde{b}^{(2)}_{p}\cdot \nabla h_p) \cdot w_{q}dx}\\
&= I_{41}+I_{42}+I_{43} 
\end{split}
\end{equation}

Next, we use commutator to further decompose $I_{41}$, 

\begin{equation}\notag
\begin{split}
 I_{41}&=\sum_{q\geq -1}\sum_{\vert p-q\vert\leq 2}{\lambda^{2s}_{q}} \int_{\mathbb{T}^n}{ \Delta_{q}( b^{(2)}_{\leq p-2}\cdot \nabla h_{p}) \cdot w_{q} dx}  \\ 
&= \sum_{q\geq -1}\sum_{\vert p-q\vert\leq 2}{\lambda^{2s}_{q}} \int_{\mathbb{T}^n}{[ \Delta_{q}\, , b^{(2)}_{\leq p-2}\cdot \nabla ] h_{p} \cdot w_{q}dx}\\
&\quad + \sum_{q\geq -1}\sum_{\vert p-q\vert\leq 2}{\lambda^{2s}_{q}} \int_{\mathbb{T}^n}{ b^{(2)}_{\leq q-2}\cdot \nabla \Delta_{q} ( h_{p}) \,w_{q}dx}  \\ 
&\quad+ \sum_{q\geq -1}\sum_{\vert p-q\vert\leq 2}{\lambda^{2s}_{q}} \int_{\mathbb{T}^n}{ \big(b^{(2)}_{\leq p-2} - b^{(2)}_{\leq q-2})\cdot \nabla \Delta_{q} ( h_{p}) \,w_{q}dx}\\
& =I_{411}+I_{412}+I_{413} 
\end{split}
\end{equation}
To simplify second term, we use $\sum_{\vert p-q\vert\leq 2}{\Delta_q (h_p)}=h_q$, and it follows that
\begin{equation}\notag
I_{412}= \sum_{q\geq -1}\sum_{\vert p-q\vert\leq 2}{\lambda^{2s}_{q}} \int_{\mathbb{T}^n}{ b^{(2)}_{\,\leq q-2}\cdot \nabla h_{q} \,w_{q}dx}=0 
\end{equation}
since $\div\,b^{(2)}_{\leq q-2}=0$.
Next, we focus on estimating $I_{411}$. Since $h\vert_{\leq Q_B}=0$, we can split $I_{411}$ using  the wavenumber 
\begin{equation}\notag
\begin{split}
 I_{411} &= \sum_{q> Q_B}\sum_{\substack{\vert p-q\vert\leq 2\\ p>Q_B+2}}{\lambda^{2s}_{q}} \int_{\mathbb{T}^n}{ [ \Delta_{q}\, , b^{(2)}_{\leq p-2}\cdot \nabla ] h_{p} \cdot w_{q} dx} \\
 &=  \sum_{q> Q_B }\sum_{\substack{\vert p-q\vert\leq 2\\ p>Q_B+2}}{\lambda^{2s}_{q}} \int_{\mathbb{T}^n}{[ \Delta_{q}\, , b^{(2)}_{\leq Q_B}\cdot \nabla ] h_{p} \cdot w_{q}dx} \\
&+ \sum_{q>Q_B}\sum_{\substack{\vert p-q\vert\leq 2\\ p>Q_B+2}}{\lambda^{2s}_{q}} \int_{\mathbb{T}^n}{[ \Delta_{q}\, , b^{(2)}_{(Q_B , p-2]}\cdot \nabla ] h_{p} \cdot w_{q}dx} =I_{4411}+I_{4412}
\end{split}
\end{equation}

\begin{equation}\notag
\begin{split}
| I_{4411} | &=\sum_{q>Q_B} \sum_{\substack{\vert p-q\vert\leq 2\\ p>Q_B+2}}{\lambda^{2s}_{q}} \int_{\mathbb{T}^n}{|\, [ \Delta_{q}\, , b^{(2)}_{\leq Q_B}\cdot \nabla ] h_{p} \cdot w_{q}\, | \, dx} \\
&\leq  \sum_{q>Q_B} \sum_{\substack{\vert p-q\vert\leq 2\\ p>Q_B+2}} { \lambda^{2s}_{q}\, || \nabla b^{(2)}_{\leq Q_B} ||_{\infty} \, || h_{p}||_{2} \, ||w_{q}||_{2} } \lesssim
 c_r\, \mu\,\sum_{q> Q_B}{\lambda^{2s+1}_{q}\,\Lambda_B \, || h_{q}||_{2}  ||w_{q}||_{2} } \\ 
&\lesssim c_{r}\, \mu\,  \sum_{q> Q_B}{ \lambda^{2s+2}_{q} \, || h_{q}||_{2} \, ||w_{q}||_{2} \, \lambda_{Q_B}\, \lambda^{-1}_{q}} 
\lesssim c_{r}\, \mu\,  \sum_{q\geq -1}{ \lambda^{2s+2}_{q} \, || h_{q}||_{2} \, ||w_{q}||_{2} } \\ 
&\lesssim c_{r}\, \mu\,  \sum_{q\geq -1}{ \lambda^{2s+2}_{q} \, || h_{q}||^{2}_{2}}
+ c_{r}\, \mu\,  \sum_{q\geq -1}{ \lambda^{2s+2}_{q} \, ||w_{q}||^{2}_{2} }
\end{split}
\end{equation}
where we used the fact that $\lambda_{Q_B}\, \lambda^{-1}_{q}<1$ for $q>Q_B$. We continue estimating $I_{4412}$, 
\begin{equation}\notag
\begin{split}
| I_{4412} |  &\leq   \sum_{q>Q_B}\sum_{\substack{\vert p-q\vert\leq 2\\ p>Q_B+2}}{\lambda^{2s}_{q}} \int_{\mathbb{T}^n}{ |\, [ \Delta_{q}\, , b^{(2)}_{(Q_B , p-2]}\cdot \nabla ] h_{p} \cdot w_{q} |\,dx} \\
& \leq \sum_{q> Q_B} \sum_{\substack{\vert p-q\vert\leq 2\\ p>Q_B+2}}  {  \lambda^{2s}_{q}\, \, ||h_p||_{2} \,||w_{q}||_{\frac{2r}{r-2}}} \, \sum_{Q_B<p^{\prime}\leq p-2} {|| \nabla b^{(2)}_{p^{\prime}} ||_{r}  }  \\ 
&\lesssim  c_{r}\, \mu\,   \sum_{q> Q_B}{ \lambda^{2s}_{q}\, ||h_q||_{2} \, \lambda^{\frac{n}{r}}_{q}\, ||w_{q}||_{2}} \, \sum_{Q_B<p^{\prime}\leq q} {\lambda^{1-\frac{n}{r}-\delta}_{p^{\prime}} \, \Lambda_{B}^{\delta} } \\ 
& \lesssim  c_{r}\, \mu\,   \sum_{q> Q_B}{ \lambda_{q}^{2s+2}\,||h_q||_{2}\, ||w_{q}||_{2} } \, \sum_{Q_B<p^{\prime}\leq q} {\lambda^{-2+\frac{n}{r} }_{q-p^{\prime}} \, \lambda^{\delta+1}_{Q_B-p^{\prime}}\,\lambda_{Q_B}^{-1} } 
\end{split}
\end{equation}
As before, assuming $\delta>-1$ and $r>\frac{n}{2}$, it follows that 
\begin{equation}\notag
\begin{split}
| I_{4412} |  & \lesssim  c_{r}\, \mu\, \sum_{q> Q_B}{ \lambda_{q}^{2s+2}\,||h_q||_{2}\, ||w_{q}||_{2} } \\ 
&\lesssim  c_{r}\, \mu\, \sum_{q\geq -1}{ \lambda_{q}^{2s+2}\,||h_q||^{2}_{2} }+ c_{r}\, \mu\, \sum_{q\geq -1}{ \lambda_{q}^{2s+2}\,||w_{q}||^{2}_{2}} 
\end{split}
\end{equation}

Next we focus on estimating $I_{413}$, 

\begin{equation}\notag
\begin{split}
 | I_{413} | &= \sum_{q> Q_B}\sum_{\substack{\vert p-q\vert\leq 2\\ p>Q_B}}{\lambda^{2s}_{q}} \int_{\mathbb{T}^n}{|\, \big(b^{(2)}_{\leq p-2} - b^{(2)}_{\leq q-2})\cdot \nabla \Delta_{q} ( h_{p}) \,w_{q}\, |\, dx}  \\ 
 &\lesssim \sum_{q> Q_B}{ \lambda_{q}^{2s+1}\, ||b^{(2)}_{(q-4\,,\,q]}\,||_{\infty} \, ||h_{q}||_{2} \, ||w_{q}||_{2}} \\ 
 & \lesssim \sum_{q> Q_B} { \lambda_{q}^{2s+1}\, ||b^{(2)}_{(q-4\,,\,Q_B]}\,||_{\infty} \, ||h_{q}||_{2} \, ||w_{q}||_{2}}\\
&\quad  +  \sum_{q> Q_B}  \sum_{\substack{q-4<p^{\prime}\leq q\\ p^{\prime}>Q_B}}{ \lambda_{q}^{2s+1}\, ||b^{(2)}_{p^{\prime}}\,||_{\infty} \, ||h_{q}||_{2} \, ||w_{q}||_{2}} \\
  &=I_{4131}+I_{4132}
\end{split}
\end{equation}
As we noted before, we use the convention that $(q-4,Q_B]$ is empty if $q-4\geq Q_B$, therefore 
 \begin{equation}\notag
  \begin{split}
  I_{4131}\lesssim c_r\, \mu\, \sum_{q> Q_B} { \lambda_{q}^{2s+1}\, ||b_{q}||_{2} \, ||w_{q}||_{2}} \lesssim c_{r}\, \mu\, \sum_{q\geq -1}{ \lambda_{q}^{2s+2}\,||h_q||^{2}_{2} }+ c_{r}\, \mu\, \sum_{q\geq -1}{ \lambda_{q}^{2s+2}\,||w_{q}||^{2}_{2}}  
  \end{split}
\end{equation}
and using definition of $\Lambda_B$, Bernstein's and Young's inequalities, we have
  \begin{equation} \notag
  \begin{split}
 I_{4132} &\lesssim c_{r}\, \mu\,  \sum_{q> Q_B}   \sum_{\substack{q-4<p^{\prime}\leq q\\ p^{\prime}>Q_B}}{ \lambda_{q}^{2s+1}\, \lambda_{p^{\prime}}^{-\delta}\,  \Lambda_B^{\delta} \, ||h_{q}||_{2} \, ||w_{q}||_{2}} \\ 
 &\lesssim  c_{r}\, \mu\,   \sum_{q> Q_B} {\lambda_{q}^{2s+2}\, ||h_{q}||_{2} \, ||w_{q}||_{2} }\, \sum_{\substack{q-4<p^{\prime}\leq q\\ p^{\prime}>Q_B}}{ \, \lambda_{p^{\prime}}^{-\delta}\, \lambda_{q}^{-1}\, \Lambda_B^{\delta}} \\ 
  &\lesssim c_{r}\, \mu\,   \sum_{q> Q_B} {\lambda_{q}^{2s+2}\, ||h_{q}||_{2} \, ||w_{q}||_{2} }\, \sum_{\substack{q-4<p^{\prime}\leq q\\ p^{\prime}>Q_B}}{ \, \lambda_{q-p^{\prime}}^{-1}\, (\lambda_{Q_B-p^{\prime}})^{1+\delta}\, \lambda_{Q_B}^{-1}} \\
  &\lesssim  c_{r}\, \mu\,   \sum_{q> Q_B} {\lambda_{q}^{2s+2}\, ||h_{q}||_{2} \, ||w_{q}||_{2} } \lesssim  c_{r}\, \mu\,   \sum_{q\geq -1} {\lambda_{q}^{2s+2}\, ||h_{q}||^2_{2} } \\
  &\quad+ \, c_{r}\, \mu\,   \sum_{q\geq -1} {\lambda_{q}^{2s+2}\, ||w_{q}||^2_{2} }\, 
\end{split}
\end{equation}
where we used $\delta>-1$.

We continue to estimate $I_{42}$. Since $h\vert_{\leq Q_B}=0$, we have

\begin{equation}\notag
\begin{split}
  | I_{42} | &\leq \sum_{q>Q_B}\sum_{\substack{\vert p-q\vert\leq 2\\ p>Q_B+2}}{\lambda^{2s}_{q}} \int_{\mathbb{T}^n}{ |\,\Delta_{q}( b^{(2)}_{p}\cdot \nabla h_{(Q\, , \, p-2\,]}\,) \cdot w_{q} |\, dx}\\ 
&\leq \sum_{q>Q_B}\sum_{\substack{\vert p-q\vert\leq 2\\ p>Q_B+2}}{\lambda^{2s}_{q}\, || b^{(2)}_{p}||_{r}\, ||w_{q}||_{\frac{2r}{r-2}}\, ||\nabla h_{(Q\, , \, p-2\,]}\,) ||_{2}} \\
&\lesssim \sum_{q>Q_B}{ \lambda^{ 2s+\frac{n}{r}}_{q} \,|| b^{(2)}_{q}||_{r} \,||w_{q}||_{2}\,} \sum_{Q_B<p^{\prime}\leq q}{\lambda_{p^{\prime}}\, ||h_{p^{\prime}} ||_{2}  }\\ 
    &\lesssim c_{r}\, \mu\,  \sum_{q>Q_B}{ \, \lambda^{s+1}_{q}\,\,||w_{q}||_{2}\, } \sum_{Q_B<p^{\prime}\leq q}{  \lambda^{s+1}_{p^{\prime}}\, ||h_{p^{\prime}} ||_{2} \, \, \lambda^{s-\delta-1}_{q}\, \lambda^{-s}_{p^{\prime}}\, \Lambda_{B}^{\delta}  } \\ 
     & \lesssim c_{r}\, \mu\,  \sum_{q>Q_B}{ \, \lambda^{s+1}_{q}\,\,||w_{q}||_{2}\, } \sum_{Q_B<p^{\prime}\leq q}{  \lambda^{s+1}_{p^{\prime}}\, ||h_{p^{\prime}} ||_{2} \, \, \lambda^{s-\delta-1}_{q-p^{\prime}}\, (\lambda_{Q_B-p^{\prime}})^{1+\delta}\, \lambda_{Q_B}^{-1}  } 
\end{split}
\end{equation}
Assuming $\delta>-1$ and using Jensen and young inequalities, it follows that 
\begin{equation}\notag
\begin{split}
   | I_{42} | & \lesssim c_{r}\, \mu\,  \sum_{q>Q_B}{ \, \lambda^{s+1}_{q}\,\,||w_{q}||_{2}\, } \sum_{Q_B<p^{\prime}\leq q}{  \lambda^{s+1}_{p^{\prime}}\, ||h_{p^{\prime}} ||_{2} \, \, \lambda^{s-\delta-1}_{q-p^{\prime}}\, } \\
   & \lesssim c_{r}\, \mu\,  \sum_{q>Q_B}{ \, \lambda^{2s+2}_{q}\,\,||w_{q}||^{2}_{2}\, }+  c_{r}\, \mu\, \sum_{q>Q_B}{\Big(  \sum_{Q_B<p^{\prime}\leq q}{  \lambda^{s+1}_{p^{\prime}}\, ||h_{p^{\prime}} ||_{2} \, \, \lambda^{s-\delta-1}_{q-p^{\prime}}\, }  \Big)^2}  \\ 
   & \lesssim c_{r}\, \mu\,  \sum_{q>Q_B}{ \, \lambda^{2s+2}_{q}\,\,||w_{q}||^{2}_{2}\, }+  c_{r}\, \mu\, \sum_{q>Q_B}{ \, \lambda^{2s+2}_{q}\,\,||h_{q}||^{2}_{2}\, } \\ 
   & \lesssim c_{r}\, \mu\,  \sum_{q\geq -1}{ \, \lambda^{2s+2}_{q}\,\,||w_{q}||^{2}_{2}\, }+  c_{r}\, \mu\, \sum_{q\geq -1}{ \, \lambda^{2s+2}_{q}\,\,||h_{q}||^{2}_{2}\, } 
\end{split}
\end{equation}
where we used $s<1+\delta$. 

Finally, we estimate of $I_{43}$ according to
\begin{equation}\notag
\begin{split}
    | I_{43}| &\leq \sum_{q\geq -1}\sum_{ p\geq q-2}{\lambda^{2s}_{q}}  \, \int_{\mathbb{T}^n}{|\,\Delta_{q}( \tilde{b}^{(2)}_{p}\cdot \nabla h_p) \cdot w_{q}\, |\, dx} \\ 
    &\leq \sum_{p>Q_B}\sum_{ q\leq p+2}{\lambda^{2s}_{q}\, || \tilde{b}^{(2)}_{p}||_{r}\, ||\nabla h_{p}||_{\frac{2r}{r-2}}\, ||w_{q} ||_{2}} \\
    &\lesssim \sum_{p>Q_B}{\lambda^{2s}_{q}\, || \tilde{b}^{(2)}_{p}||_{r}\,\lambda^{\frac{n}{r}+1}_{p}\, || h_{p}||_{2}}\sum_{ q\leq p+2}{\, ||w_{q} ||_{2}} \\
    &\lesssim c_{r}\, \mu \, \sum_{p>Q_B}{ \lambda^{s+1}_{p}\,  || h_{p}||_{2}\,  \lambda^{-s-\delta}_{p}\, \Lambda^{\delta}_{B}\, } \sum_{ Q_B<q\leq p+2}{ \lambda^{s+1}_{q}\, ||w_{q} ||_{2} \, \, \lambda^{s-1}_{q}\, } \\ 
    &\lesssim c_{r}\, \mu \, \sum_{p>Q_B}{ \lambda^{s+1}_{p}\,  || h_{p}||_{2}}\sum_{ Q_B<q\leq p+2}{ \lambda^{s+1}_{q}\, ||w_{q} ||_{2} \, \, \lambda^{s+\delta}_{q-p}\, (\lambda_{Q_B-q})^{1+\delta}\, \lambda_{Q_B}^{-1}}
\end{split}
\end{equation}
where assuming $\delta>-1$ as before, gives
\begin{equation}\notag
\begin{split}
    | I_{43}| &\lesssim c_{r}\, \mu \, \sum_{p>Q_B}{ \lambda^{s+1}_{p}\,  || h_{p}||_{2}}\sum_{ Q_B<q\leq p+2}{ \lambda^{s+1}_{q}\, ||w_{q} ||_{2} \, \, \lambda^{s+\delta}_{q-p}} \\ 
    &\lesssim c_{r}\, \mu \, \sum_{p>Q_B}{ \lambda^{2s+2}_{p}\,  || h_{p}||^{2}_{2}} + c_{r}\, \mu \,\sum_{p>Q_B} \Big( \sum_{ Q_B<q\leq p+2}{ \lambda^{s+1}_{q}\, ||w_{q} ||_{2} \, \, \lambda^{s+\delta}_{q-p}} \Big)^2 \\ 
    &\lesssim c_{r}\, \mu \, \sum_{p>Q_B}{ \lambda^{2s+2}_{p}\,  || h_{p}||^{2}_{2}} + c_{r}\, \mu \,\sum_{p>Q_B} { \lambda^{2s+2}_{p}\, ||w_{p} ||^{2}_{2}} \\ 
    &\lesssim c_{r}\, \mu \, \sum_{p\geq -1}{ \lambda^{2s+2}_{p}\,  || h_{p}||^{2}_{2}} + c_{r}\, \mu \,\sum_{p\geq -1} { \lambda^{2s+2}_{p}\, ||w_{p} ||^{2}_{2}} 
\end{split}
\end{equation}
where we used $s>-\delta$, and Jensen and young inequalities.

Consequently, taken together all the estimates, we arrive at 
\begin{equation}\label{I_4}
I_4\lesssim  c_{r}\, \mu \, \sum_{q\geq -1}{ \lambda^{2s+2}_{q}\,  || w_{q}||^{2}_{2}} + c_{r}\, \mu \,\sum_{q\geq -1} { \lambda^{2s+2}_{q}\, ||h_{q} ||^{2}_{2}} 
\end{equation}
\vspace{0.5cm}

\subsubsection{ Estimate for $I_{5}$}

Similarly, we can decompose $I_{5}$ using Bony's paraproduct 
\begin{equation}\notag
\begin{split}
I_{5} &= -\sum_{q\geq -1}{\lambda^{2s}_{q}} \int_{\mathbb{T}^3}{\Delta_{q}(u\cdot \nabla h)\cdot h_{q}dx}\\
=&-\sum_{q\geq -1}\sum_{\vert p-q\vert\leq 2}{\lambda^{2s}_{q}} \int_{\mathbb{T}^n}{\Delta_{q}( u_{\leq p-2}\cdot \nabla h_p) \cdot h_{q}dx} \\ 
&\quad- \sum_{q\geq -1}\sum_{\vert p-q\vert\leq 2}{\lambda^{2s}_{q}} \int_{\mathbb{T}^n}{\Delta_{q}( u_{p}\cdot \nabla h_{\leq p-2}) \cdot h_{q}dx} \\ 
&\quad- \sum_{q\geq -1}\sum_{ p\geq q-2}{\lambda^{2s}_{q}} \int_{\mathbb{T}^n}{\Delta_{q}( \tilde{u}_{p}\cdot \nabla h_p) \cdot h_{q}dx}\\
&=I_{51}+I_{52}+I_{53} \
\end{split}
\end{equation}
We use commutator to further decompose $I_{51}$ according to 
\begin{equation}\notag
\begin{split}
I_{51} =& -\sum_{q\geq -1}\sum_{\vert p-q\vert\leq 2}{\lambda^{2s}_{q}} \int_{\mathbb{T}^n}{[\,\Delta_{q}\, , \,  u_{\leq p-2}\cdot \nabla \, ] \, h_p \, h_{q}dx} \\ 
&\quad-\sum_{q\geq -1}\sum_{\vert p-q\vert\leq 2}{\lambda^{2s}_{q}} \int_{\mathbb{T}^n}{u_{\leq q-2}\cdot \nabla\Delta_{q} h_p\, h_{q}dx} \\ 
&\quad-\sum_{q\geq -1}\sum_{\vert p-q\vert\leq 2}{\lambda^{2s}_{q}} \int_{\mathbb{T}^n}{(u_{\leq p-2}-u_{\leq q-2})\cdot \Delta_{q} h_{p}\, h_{q} dx}\\
&= I_{511}+I_{512}+I_{513} 
\end{split}
\end{equation}
Note that $I_{512}$ vanishes since $\nabla\cdot u_{\leq q-2}=0$. Next, we continue to estimate $I_{511}$. Utilizing the wavenumber $Q_{U}(t)$, we may first split $I_{511}$ into three terms and estimate each one separately,
\begin{equation}\notag
\begin{split}
| I_{511} |  &\leq \sum_{q\geq -1}\sum_{\vert p-q\vert\leq 2}{\lambda^{2s}_{q}} \int_{ \mathbb{T}^n}{ |\,[\Delta_{q}\, , \,  u_{\leq p-2}\cdot \nabla \, ] \, h_p \, h_{q} |\,dx} \\ 
&\leq \sum_{-1\leq p\leq Q_U+2}\sum_{\vert p-q\vert\leq 2}{\lambda^{2s}_{q}} \int_{\mathbb{T}^n}{|\, [\Delta_{q}\, , \,  u_{p-2}\cdot \nabla \, ] \, h_p \, h_{q} | \, dx} \\ 
&\quad+\sum_{ p> Q_U+2}\sum_{\vert p-q\vert\leq 2}{\lambda^{2s}_{q}} \int_{\mathbb{T}^n}{|\, [\,\Delta_{q}\, , \,  u_{\leq Q_U}\cdot \nabla \, ] \, h_p \, h_{q} |\,dx} \\ 
&\quad+\sum_{p>Q_U+2}\sum_{\vert p-q\vert\leq 2}{\lambda^{2s}_{q}} \int_{\mathbb{T}^n}{ |\, [\,\Delta_{q}\, , \,  u_{(\, Q_U\, , \, p-2\, ]}\cdot \nabla \, ] \, h_p \, h_{q}|\,dx}\\
&=I_{5111}+I_{5112}+I_{5113} 
\end{split}
\end{equation}

\begin{equation}\notag
\begin{split}
I_{5111} &\leq \sum_{-1\leq p\leq Q_U+2}\sum_{\vert p-q\vert \leq 2}{\lambda^{2s}_{q}\, ||h_q ||_{2}\, ||h_p ||_{2}\, || \nabla u_{\leq p-2}||_{\infty}\,}\\
&\leq \sum_{-1\leq p\leq Q_U+2} {\lambda^{2s}_{p}\, ||h_p ||_{2}^2\, ||\nabla u_{\leq Q_U+2}||_{\infty}\, }\\
&\lesssim  c_r\, \nu\sum_{-1\leq p\leq Q_U+2} {\lambda^{2s+2}_{p}\, ||h_p ||_{2}^2\,\,( \lambda_{Q_U+2-p})^2\,} 
\end{split}
\end{equation}
even though  $( \lambda_{Q_U+2-p})^2\geq 1$ for all $p\leq Q_U+2$, the last expression above still makes sense as we are summing over a finite number of modes. Since $( \lambda_{Q_U+2-p})^2$ is bounded by $\lambda^2_{Q_U+3}$, a positive constant number, we can absorbed it into $\lesssim$. Thus, 
\begin{equation}\notag
\begin{split}
I_{5111} \lesssim  c_r\, \nu\, \sum_{-1\leq p\leq Q_U+2} {\lambda^{2s+2}_{p}\, ||h_p ||_{2}^2\,} \lesssim c_r\, \nu\, \, \sum_{p\geq -1} {\lambda^{2s+2}_{p}\, ||h_p ||_{2}^2\,} 
\end{split}
\end{equation}
\begin{equation}\notag
\begin{split}
    I_{5112} & \leq \sum_{p>Q_U+2}\, \sum_{\vert p-q\vert\leq 2}{\lambda^{2s}_{q}\, ||h_q ||_{2}\, ||h_p ||_{2}\, || \nabla u_{\leq Q_U}||_{\infty}\,} \\
&\lesssim \sum_{q>Q_U}{\lambda^{2s+2}_{q}\, ||h_q ||_{2}\, ||h_q ||_{2} \,\, \lambda_{q}^{-2} || \nabla u_{\leq Q_U}||_{\infty}\,} \\ 
& \lesssim c_r\, \nu\, \lesssim \sum_{q>Q_U}{\lambda^{2s+2}_{q}\, ||h_q ||^{2}_{2}\, \lambda_{q}^{-2} \,\Lambda_U^2 } \lesssim c_{r}\, \mu \,  \sum_{q>Q_U}{\lambda^{2s+2}_{q}\, ||h_q ||^{2}_{2} \,\, \lambda_{q-Q_U}^{-2} } \\
&\lesssim c_{r}\, \mu \,  \sum_{q>Q_U}{\lambda^{2s+2}_{q}\, ||h_q ||^{2}_{2} \, } \lesssim c_{r}\, \mu \,  \sum_{q\geq -1}{\lambda^{2s+2}_{q}\, ||h_q ||^{2}_{2} } 
\end{split}
\end{equation}
We carry on with estimate of $ I_{5113}$ according to
\begin{equation}\notag
\begin{split}
    I_{5113} &\leq \sum_{p>Q_U+2}\sum_{\substack{\vert p-q\vert\leq 2\\  q>Q_U}}{\lambda^{2s}_{q}\, ||h_q ||_{2}\, ||h_p ||_{2}\, || \nabla u_{(\, Q_U\, , \, p-2\, ]} ||_{\infty}\,}\\
    &\lesssim  \sum_{p>Q_U+2}\sum_{\substack{\vert p-q\vert\leq 2\\  q>Q_U}} {\lambda^{2s}_{q}\, ||h_q ||_{2}\, ||h_p ||_{2}} \sum_{Q_U<p^{\prime}\leq p-2}{ \lambda_{p^{\prime}}\, || u_{p^{\prime}}||_{\infty}\,} \\
  &\lesssim  \sum_{q>Q_U} {\lambda^{2s}_{q}\, ||h_q ||^2_{2}} \sum_{Q_U<p^{\prime}\leq q}{ \lambda_{p^{\prime}}\, || u_{p^{\prime}}||_{\infty}\,} 
  \end{split}
\end{equation}
and using the definition of wavenumber again leads to
\begin{equation}\notag
\begin{split}
 I_{5113}&\lesssim  c_{r}\, \nu\, \sum_{q>Q_U} {\lambda^{2s}_{q}\, ||h_q ||^{2}_{2}\,} \sum_{Q_U<p^{\prime}\leq q}{ \lambda_{p^{\prime}}\, \Lambda^{1+\sigma}_U\, \lambda^{-\sigma}_{p^{\prime}}  \,}\\
%\lesssim c_{r}\, \nu\,  \sum_{q>Q} {\lambda^{2s+2}_{q}\, ||b_q ||^{2}_{2}\,} \,\sum_{Q<p^{\prime}\leq q}{ \lambda^{1-\sigma}_{p^{\prime}} \, \lambda_{Q}^{1+\sigma}\, \lambda^{-2}_{q}} 
&\lesssim c_{r}\, \nu\,  \sum_{q>Q_U} {\lambda^{2s+2}_{q}\, ||h_q ||^{2}_{2}\,} \,\sum_{Q_U<p^{\prime}\leq q}{ (\lambda_{Q_U-p^{\prime}})^{1+\sigma}\, \lambda^{-2}_{q-p^{\prime}}} \\
&\lesssim c_{r}\, \nu\,  \sum_{q>Q_U} {\lambda^{2s+2}_{q}\, ||h_q ||^{2}_{2}\,} \,\sum_{Q_U<p^{\prime}\leq q}{ (\lambda_{Q_U-p^{\prime}})^{1+\sigma}\, } \\
&\lesssim c_{r}\, \nu\,  \sum_{q>Q_U} {\lambda^{2s+2}_{q}\, ||h_q ||^{2}_{2}\,}  \lesssim \, c_{r}\, \nu\,  \sum_{q\geq -1} {\lambda^{2s+2}_{q}\, ||h_q ||^{2}_{2}\,} 
\end{split}
\end{equation}
where once more we used $1+\sigma \geq 0$ which implies that $\sigma \geq -1$. \\

Next, we continue estimating $I_{513}$ again by splitting it into high and low modes and estimate each one
independently, i.e,
\begin{equation}\notag
\begin{split}
    |I_{513}|&\leq \sum_{q\geq -1}\sum_{\vert p-q\vert\leq 2} {\lambda^{2s}_{q}} \int_{\mathbb{T}^n}{| (u_{\leq p-2}-u_{\leq q-2})\cdot \Delta_{q} h_{p}\, h_{q} |\, dx} \\
    &\leq \sum_{q\geq -1}\sum_{\vert p-q\vert\leq 2}{\lambda^{2s}_{q}\, ||u_{\leq p-2}-u_{\leq q-2}||_{\infty} \,    || \nabla h_p ||_{2}\, ||h_q ||_{2}\, }\\
    &\lesssim \sum_{q\geq -1}{\lambda^{2s+1}_{q}\,  ||u_{\, (q-4\, , \, q\, ] \, } ||_{\infty} \,||h_q ||^{2}_{2} \, } \lesssim \sum_{-1\leq q\leq Q_U}{\lambda^{2s+1}_{q}\, ||u_{\, (q-4\, , \, q\, ] \, } ||_{\infty} \, ||h_q ||^{2}_{2}\, }\\
    &\quad+ \sum_{q>Q_U}{\lambda^{2s+1}_{q}\, ||u_{\, (q-4\, , \, Q_U\, ] \, } ||_{\infty} \, ||h_q ||^{2}_{2}\, } + \sum_{q>Q_U} \sum_{\substack{q-4<p^{\prime}\leq q \\ p^{\prime}>Q_U}}{\lambda^{2s+1}_{q}\, ||u_{p^{\prime}}||_{\infty} \, ||h_q ||^{2}_{2} \, }\\
&    =I_{5131}+I_{5132}+I_{5133} 
\end{split}
\end{equation}
\begin{equation}\notag
\begin{split}
    I_{5131}&=\sum_{-1\leq q\leq Q_U}{\lambda^{2s+1}_{q}\, ||u_{\, (q-4\, , \, q\, ] \, } ||_{\infty} \, ||h_q ||^{2}_{2}\, } \lesssim \sum_{-1\leq q\leq Q_U}{\lambda^{2s+1}_{q}\, ||u_{\leq Q_U } ||_{\infty} \, ||h_q ||^{2}_{2}\, } \\
    & \lesssim c_r\, \nu\, \sum_{-1\leq q\leq Q_U}{\lambda^{2s+2}_{q}\, ||h_q ||^{2}_{2}\, \Lambda_U\, \lambda_q^{-1}} \lesssim c_r\, \nu\, \sum_{-1\leq q\leq Q_U}{\lambda^{2s+2}_{q}\, ||h_q ||^{2}_{2}\, \lambda_{Q_U-q}}   \\ 
&\lesssim  c_{r}\, \nu \, \sum_{-1\leq q\leq Q_U}{\lambda^{2s+2}_{q}\, ||h_q ||^{2}_{2}\, } \lesssim c_{r}\, \nu \, \sum_{q\geq -1}{\lambda^{2s+2}_{q}\, ||h_q ||^{2}_{2}\, }
\end{split}
\end{equation}
where we used a similar argument as the one in the estimation of $I_{5111}$. Indeed, since $\lambda_{Q_U-q}\leq \lambda_{Q_U+1}$, is bounded by a constant number, it can be absorbed to the constants in the estimate. 

As before, we can adopt the convention that $(q-4,Q_U]$ is empty if $q-4\geq Q_U$. It follows that 
\begin{equation}\notag
\begin{split}
    I_{5132}&= \sum_{q>Q_U}{\lambda^{2s+1}_{q}\, ||u_{\, (q-4\, , \, Q_U\, ] \, } ||_{\infty} \, ||h_q ||^{2}_{2}\,  } 
    \lesssim  \sum_{q>Q_U}{\lambda^{2s+1}_{q}\, ||u_{\leq Q } ||_{\infty} \, ||h_q ||^{2}_{2}\,  }  \\ 
    &\lesssim  c_{r}\, \nu \, \sum_{q>Q_U}{\lambda^{2s+2}_{q}\, ||h_q ||^{2}_{2}\, \, \Lambda_U \lambda^{-1}_{q} } 
      \lesssim  c_{r}\, \nu \, \sum_{q>Q_U}{\lambda^{2s+2}_{q}\, ||h_q ||^{2}_{2}\, \, (\lambda_{Q_U-q})^{-1} } \\
&\lesssim  c_{r}\, \nu \, \sum_{q>Q_U}{\lambda^{2s+2}_{q}\, ||h_q ||^{2}_{2}\, } \lesssim c_{r}\, \nu \, \sum_{q\geq -1}{\lambda^{2s+2}_{q}\, ||h_q ||^{2}_{2}\, }
\end{split}
\end{equation}
and estimating the third term yields 
\begin{equation}\notag
\begin{split}
    I_{5133} & \lesssim \sum_{q>Q_U}{\lambda^{2s+1}_{q}\, ||h_q ||^{2}_{2}\, } \,\sum_{\substack{q-4<p^{\prime}\leq q\\  p^{\prime}>Q_U}}{ || u_{p^{\prime}}||_{\infty}\,} \\ 
  &\lesssim  c_{r}\, \nu \sum_{q>Q_U} {\lambda^{2s+2}_{q}\, ||h_q ||^{2}_{2}\,} \sum_{\substack{q-4<p^{\prime}\leq q\\ p^{\prime}>Q_U}}{ \lambda^{-1}_{q}\, \Lambda_U^{1+\sigma}\, \lambda^{-\sigma}_{p^{\prime}}\, }\\
  &\lesssim  c_{r}\, \nu \, \sum_{q>Q_U} {\lambda^{2s+2}_{q}\, ||h_q ||^{2}_{2}\,} \sum_{\substack{q-4<p^{\prime}\leq q\\  p^{\prime}>Q_U}}{  (\lambda_{Q_U-p^{\prime}})^{1+\sigma} \lambda^{-1}_{q-p^{\prime}}\, } \\ 
&\lesssim  c_{r}\, \nu \, \sum_{q>Q_U} {\lambda^{2s+2}_{q}\, ||h_q ||^{2}_{2}\,} 
\lesssim  c_{r}\, \nu \, \sum_{q\geq -1} {\lambda^{2s+2}_{q}\, ||h_q ||^{2}_{2}\,} 
\end{split}
\end{equation}
where $\sigma\geq-1$. 
 
Next, we proceed by computing the estimate of $I_{52}$. First, let's start by separating the high and low frequencies,
\begin{equation}\notag
\begin{split}
 |I_{52}|& \leq \sum_{q\geq -1}\sum_{\vert p-q\vert\leq 2}{\lambda^{2s}_{q}} \int_{\mathbb{T}^n}{ |\, \Delta_{q}( u_{p}\cdot \nabla h_{\leq p-2}) \cdot h_{q} | \, dx} \\
 &\leq \sum_{-1\leq q\leq Q_U}\sum_{\vert p-q\vert\leq 2}{\lambda^{2s}_{q}} \int_{\mathbb{T}^n}{ |\, \Delta_{q}( u_{p}\cdot \nabla h_{\leq p-2}) \cdot h_{q} | \, dx} \\
 &\quad+  \sum_{q> Q_U}\sum_{\vert p-q\vert\leq 2}{\lambda^{2s}_{q}} \int_{\mathbb{T}^n}{ |\, \Delta_{q}( u_{p}\cdot \nabla h_{(Q_U\, , \, p-2]\,}) \cdot h_{q} | \, dx}\\
 &\quad+  \sum_{q> Q_U}\sum_{\vert p-q\vert\leq 2}{\lambda^{2s}_{q}} \int_{\mathbb{T}^n}{ |\, \Delta_{q}( u_{p}\cdot \nabla h_{\leq Q_U}) \cdot h_{q} | \, dx} \\
 &=I_{521}+I_{522}+I_{523}
\end{split}
\end{equation}
 %%****Above, note that the reason we seperated high modes into two terms I_522 and I_523 is that we can get better results, we'll show the results for s<\sigma+1. However, it is possible to estimate high modes without seperating them into two terms, then we can only show it for s<0 which is very restrictive. 
Since $I_{521}$ only consists of low modes, we can estimate it analogous to $I_{5111}$. Indeed we have,
\begin{equation}\notag
\begin{split}
I_{521} &\leq \sum_{-1\leq q \leq Q_U}\sum_{\vert p-q\vert\leq 2}{\lambda^{2s}_{q}\, ||u_{p}||_{\infty} \,    || \nabla h_{\leq p-2\,} ||_{2}\, ||h_q ||_{2}\, }\\
%& \lesssim \sum_{-1\leq q\leq Q_U}\sum_{\vert p-q\vert\leq 2}{\lambda^{2s}_{q}\, ||u_{p}||_{\infty} \, ||h_q ||_{2}\, } \sum_{-1\leq p^{\prime}\leq p-2} {\lambda_{p^{\prime}} \,|| h_{p^{\prime}} ||_{2} }
&\lesssim \sum_{-1\leq q\leq Q_U} {\lambda^{2s}_{q}\, ||u_{q}||_{\infty} \, ||h_q ||_{2}\, } \sum_{-1\leq p^{\prime}\leq q} {\lambda_{p^{\prime}} \,|| h_{p^{\prime}} ||_{2} }\\
&\leq \sum_{-1\leq q\leq Q_U} {\lambda^{s+1}_{q}\, ||u_{\leq Q_U}||_{\infty} \, ||h_q ||_{2}\, } \sum_{-1\leq p^{\prime}\leq q} {\lambda^{s+1}_{p^{\prime}} \,|| h_{p^{\prime}} ||_{2} \, \lambda^{s-1}_{q}\, \lambda^{-s}_{p^{\prime}}}\\
%&\lesssim c_r\, \nu \sum_{-1\leq q\leq Q_U} {\lambda^{s+1}_{q}\, ||h_q ||_{2}\, \lambda_{Q_U}} \sum_{-1\leq p^{\prime}\leq q} {\lambda^{s+1}_{p^{\prime}} \,|| h_{p^{\prime}} ||_{2} \, \lambda^{s-1}_{q-p^{\prime}}\, \lambda^{-1}_{p^{\prime}} }\\
&\lesssim c_r\, \nu \sum_{-1\leq q\leq Q_U} {\lambda^{s+1}_{q}\, ||h_q ||_{2}\, \lambda_{Q_U}} \sum_{-1\leq p^{\prime}\leq q} {\lambda^{s+1}_{p^{\prime}} \,|| h_{p^{\prime}} ||_{2} \, \lambda^{s-1}_{q-p^{\prime}}\, \lambda^{-1}_{p^{\prime}} }\\
\end{split}
\end{equation}
%%% ****** If we want, we could intentionally keep the range of s to be larger than s<1, i.e s<2), and this is always possible here thanks to the fact that we only sum over finite number of modes here.******* and for $s<1$, we further use Young and Jensen’s inequalities to get

and for $s<1$, we further use Young and Jensen’s inequalities to get
\begin{equation}\notag
\begin{split}
I_{521} %&\lesssim c_{r}\, \nu\,  \sum_{-1\leq q\leq Q_U}  {\lambda^{2s+2}_{q}\,  ||h_q ||^{2}_{2}\, \lambda^2_{Q_U}} +  c_{r}\, \nu\, \sum_{-1\leq q\leq Q_U} \Big( \sum_{-1<p^{\prime}\leq q} {\lambda^{s+1}_{p^{\prime}} \,|| h_{p^{\prime}} ||_{2} \, \big( \lambda_{q-p^{\prime}} \big)^{s-1} \, } \Big)^{2}\,\\ 
&\lesssim c_{r}\, \nu\,  \sum_{-1\leq q\leq Q_U}  {\lambda^{2s+2}_{q}\,  ||h_q ||^{2}_{2}\, \lambda^2_{Q_U}} \\
&\quad+  c_{r}\, \nu\, \sum_{-1\leq q\leq Q_U} \Big( \sum_{-1\leq p^{\prime}\leq q} {\lambda^{s+1}_{p^{\prime}} \,|| h_{p^{\prime}} ||_{2} \, \big( \lambda_{q-p^{\prime}} \big)^{s-1} \, } \Big)^{2}\,\\ 
&\lesssim c_{r}\, \nu\,  \sum_{-1\leq q\leq Q_U}  {\lambda^{2s+2}_{q}\,  ||h_q ||^{2}_{2}} +  c_{r}\, \nu\, \sum_{-1\leq q\leq Q_U} {\lambda^{2s+2}_{p^{\prime}} \,|| h_{p^{\prime}} ||^{2}_{2}} \sum_{-1\leq p^{\prime}\leq q}{  \lambda^{s-1}_{q-p^{\prime}}} \,\\ 
&\lesssim c_{r}\, \nu\,  \sum_{-1\leq q\leq Q_U}  {\lambda^{2s+2}_{q}\,  ||h_q ||^{2}_{2}\,} \lesssim c_{r}\, \nu\,  \sum_{q\geq -1}  {\lambda^{2s+2}_{q}\,  ||h_q ||^{2}_{2}}
\end{split}
\end{equation}
where we absorbed $\lambda^2_{Q_U}$ into constants.

Utilizing the definition of wavenumber $\Lambda_U$, and carrying on estimates, we have
\begin{equation}\notag
\begin{split}
  I_{522} &=  \sum_{q> Q_U}\sum_{\substack{\vert p-q\vert\leq 2\\ p>Q_U+2}}{\lambda^{2s}_{q}} \int_{\mathbb{T}^n}{ |\, \Delta_{q}( u_{p}\cdot \nabla h_{(Q\, , \, p-2]\,}) \cdot h_{q} | \, dx} \\  
  &\lesssim \sum_{q>Q_U}\sum_{\substack{\vert p-q\vert\leq 2\\ p>Q_U+2}}{\lambda^{2s}_{q}\, ||u_{p}||_{\infty} \,    || \nabla h_{(Q_U\, , \, p-2]\,} ||_{2}\, ||h_q ||_{2}\, }\\
 %\lesssim c_{r}\, \nu\,  \sum_{q>Q} \sum_{\vert p-q\vert\leq 2\, \, p>Q+2} {\lambda^{2s}_{q}\, \Lambda^{1+\sigma} \, \lambda^{-\sigma}_{p}\,   ||b_q ||_{2} \, || \nabla b_{(Q\, , \, p-2]\,} ||_{2} } \,\\ \notag
&\lesssim c_{r}\, \nu\,  \sum_{q>Q_U}\sum_{\substack{\vert p-q\vert\leq 2\\ p>Q_U+2}}{\lambda^{2s}_{q}\, \Lambda_{U}^{1+\sigma} \, \lambda^{-\sigma}_{p}\,   ||h_q ||_{2}} \, \sum_{Q_U<p^{\prime}\leq p-2} {\lambda_{p^{\prime}} \,|| h_{p^{\prime}} ||_{2} } \,\\ 
&\lesssim c_{r}\, \nu\,  \sum_{q>Q_U}  {\lambda^{2s-\sigma}_{q}\, \lambda_{Q_U}^{1+\sigma} \, ||h_q ||_{2}} \, \sum_{Q_U<p^{\prime}\leq q} {\lambda^{s+1}_{p^{\prime}} \,|| h_{p^{\prime}} ||_{2} \, \, \lambda^{-s}_{p^{\prime}} \, } \,\\ 
%\lesssim c_{r}\, \nu\,  \sum_{q>Q}  {\lambda^{s+1}_{q}\,  ||b_q ||_{2}} \, \sum_{Q<p^{\prime}\leq q} {\lambda^{s+1}_{p^{\prime}} \,|| b_{p^{\prime}} ||_{2} \, \lambda^{s-\sigma-1}_{q} \, \lambda_{Q}^{1+\sigma} \, \,\lambda^{-s}_{p^{\prime}} \, } \,\\ 
&\lesssim c_{r}\, \nu\,  \sum_{q>Q_U}  {\lambda^{s+1}_{q}\,  ||h_q ||_{2}} \, \sum_{Q_U<p^{\prime}\leq q} {\lambda^{s+1}_{p^{\prime}} \,|| h_{p^{\prime}} ||_{2} \, \big( L \lambda_{q-p^{\prime}} \big)^{s-\sigma-1} \, \lambda_{Q_U-p^{\prime}}^{1+\sigma} \, }  
\end{split}
\end{equation}
now assuming $\sigma \geq -1$, and using Jensen's and Young's inequalities yields
\begin{equation}\notag
\begin{split}
I_{522} &\lesssim c_{r}\, \nu\,  \sum_{q>Q_U}  {\lambda^{2s+2}_{q}\,  ||h_q ||^{2}_{2}} +  c_{r}\, \nu\, \sum_{q>Q_U} \Big( \sum_{Q_U<p^{\prime}\leq q} {\lambda^{s+1}_{p^{\prime}} \,|| h_{p^{\prime}} ||_{2} \, \big( L \lambda_{q-p^{\prime}} \big)^{s-\sigma-1} \, } \Big)^{2}\,\\ 
%\lesssim c_{r}\, \nu\,  \sum_{q>Q_U}  {\lambda^{2s+2}_{q}\,  ||h_q ||^{2}_{2}} +  c_{r}\, \nu\, \sum_{q>Q_U} \sum_{Q_U<p^{\prime}\leq q} {\lambda^{2s+2}_{p^{\prime}} \,|| h_{p^{\prime}} ||^2_{2} \, \big( L \lambda_{q-p^{\prime}} \big)^{s-\sigma-1} \, }\,\\ 
%\lesssim c_{r}\, \nu\,  \sum_{q>Q_U}  {\lambda^{2s+2}_{q}\,  ||h_q ||^{2}_{2}} +  c_{r}\, \nu\, \sum_{q>Q_U} {\lambda^{2s+2}_{p^{\prime}} \,|| h_{p^{\prime}} ||^2_{2}}  \sum_{Q_U<p^{\prime}\leq q}  {\big( L \lambda_{q-p^{\prime}} \big)^{s-\sigma-1} \, }\,\\ 
&\lesssim c_{r}\, \nu\,  \sum_{q>Q_U}  {\lambda^{2s+2}_{q}\,  ||h_q ||^{2}_{2}} \lesssim c_{r}\, \nu\,  \sum_{q\geq -1}  {\lambda^{2s+2}_{q}\,  ||h_q ||^{2}_{2}} 
\end{split}
\end{equation}
where we used $s<\sigma+1$. Similarly for $I_{523}$, we have
\begin{equation}\notag
\begin{split}
  I_{523} &=  \sum_{q> Q_U}\sum_{\substack{\vert p-q\vert\leq 2\\ p>Q_U+2}}{\lambda^{2s}_{q}} \int_{\mathbb{T}^n}{ |\, \Delta_{q}( u_{p}\cdot \nabla h_{\leq Q_U}) \cdot h_{q} | \, dx} \\  
  &\lesssim \sum_{q>Q_U}\sum_{\vert p-q\vert\leq 2}{\lambda^{2s}_{q}\, ||u_{p}||_{\infty} \,    || \nabla h_{\leq Q_U} ||_{2}\, ||h_q ||_{2}\, }\\
%&\lesssim c_{r}\, \nu\,  \sum_{q>Q_U}\sum_{\vert p-q\vert\leq 2}{\lambda^{2s}_{q}\, \Lambda_{U}^{1+\sigma} \, \lambda^{-\sigma}_{p}\,   ||h_q ||_{2}} \, \sum_{p^{\prime}\leq Q_U} {\lambda_{p^{\prime}} \,|| h_{p^{\prime}} ||_{2} } \,\\ 
&\lesssim c_{r}\, \nu\,  \sum_{q>Q_U}  {\lambda^{2s-\sigma}_{q}\, \lambda_{Q_U}^{1+\sigma} \, ||h_q ||_{2}} \, \sum_{p^{\prime}\leq Q_U} {\lambda_{p^{\prime}} \,|| h_{p^{\prime}} ||_{2}  } \,\\ 
%\lesssim c_{r}\, \nu\,  \sum_{q>Q}  {\lambda^{s+1}_{q}\,  ||b_q ||_{2}} \, \sum_{Q<p^{\prime}\leq q} {\lambda^{s+1}_{p^{\prime}} \,|| b_{p^{\prime}} ||_{2} \, \lambda^{s-\sigma-1}_{q} \, \lambda_{Q}^{1+\sigma} \, \,\lambda^{-s}_{p^{\prime}} \, } \,\\ 
&\lesssim c_{r}\, \nu\,  \sum_{q>Q_U}  {\lambda^{s+1}_{q}\,  ||h_q ||_{2}} \, \sum_{-1\leq p^{\prime}\leq Q_U} {\lambda^{s+1}_{p^{\prime}} \,|| h_{p^{\prime}} ||_{2} \, \big( L \lambda_{q-p^{\prime}} \big)^{s-\sigma-1} \, \lambda_{Q_U-p^{\prime}}^{1+\sigma} \, }  
\end{split}
\end{equation}
For $\sigma \geq -1$, we point out that $\lambda_{Q_U-p^{\prime}}^{1+\sigma}\leq \lambda_{Q_U+1}^{1+\sigma}$, is bounded by a constant number, and we can move it out and absorb it into $\lesssim$. Once again, implementing the condition $s<\sigma+1$ and utilizing Jensen's and Young's inequalities leads to 
\begin{equation}\notag
\begin{split}
I_{523} %\lesssim c_{r}\, \nu\,  \sum_{q>Q_U}  {\lambda^{2s+2}_{q}\,  ||h_q ||^{2}_{2}} +  c_{r}\, \nu\, \sum_{q>Q_U} \Big( \sum_{Q_U<p^{\prime}\leq q} {\lambda^{s+1}_{p^{\prime}} \,|| h_{p^{\prime}} ||_{2} \, \big( L \lambda_{q-p^{\prime}} \big)^{s-\sigma-1} \, } \Big)^{2}\,\\ 
%\lesssim c_{r}\, \nu\,  \sum_{q>Q_U}  {\lambda^{2s+2}_{q}\,  ||h_q ||^{2}_{2}} +  c_{r}\, \nu\, \sum_{q>Q_U} \sum_{Q_U<p^{\prime}\leq q} {\lambda^{2s+2}_{p^{\prime}} \,|| h_{p^{\prime}} ||^2_{2} \, \big( L \lambda_{q-p^{\prime}} \big)^{s-\sigma-1} \, }\,\\ 
%\lesssim c_{r}\, \nu\,  \sum_{q>Q_U}  {\lambda^{2s+2}_{q}\,  ||h_q ||^{2}_{2}} +  c_{r}\, \nu\, \sum_{q>Q_U} {\lambda^{2s+2}_{p^{\prime}} \,|| h_{p^{\prime}} ||^2_{2}}  \sum_{Q_U<p^{\prime}\leq q}  {\big( L \lambda_{q-p^{\prime}} \big)^{s-\sigma-1} \, }\,\\ 
\lesssim c_{r}\, \nu\,  \sum_{q>Q_U}  {\lambda^{2s+2}_{q}\,  ||h_q ||^{2}_{2}} \lesssim c_{r}\, \nu\,  \sum_{q\geq -1}  {\lambda^{2s+2}_{q}\,  ||h_q ||^{2}_{2}} 
\end{split}
\end{equation}

Proceeding with the estimate of $I_{53}$ and separating high and low frequency modes, we have 
\begin{equation}\notag
\begin{split}
|I_{53}| &\leq \sum_{q\geq -1}\sum_{ p\geq q-2}{\lambda^{2s}_{q}} \int_{\mathbb{T}^n}{ |\, \Delta_{q}( \tilde{u}_{p}\cdot \nabla h_{p}) \cdot h_{q} | \, dx} \\
&\leq \sum_{-1\leq q\leq Q_U}\sum_{ p\geq q-2}{\lambda^{2s}_{q}} \int_{\mathbb{T}^n}{ |\, \Delta_{q}( \tilde{u}_{p}\cdot \nabla h_{p}) \cdot h_{q} | \, dx} \\
& + \sum_{q> Q_U}\sum_{ p\geq q-2}{\lambda^{2s}_{q}} \int_{\mathbb{T}^n}{ |\, \Delta_{q}( \tilde{u}_{p}\cdot \nabla h_{p}) \cdot h_{q} | \, dx}=I_{531}+I_{532} 
\end{split}
\end{equation}
\begin{equation}\notag
\begin{split}
I_{531}&\leq \sum_{-1\leq q\leq Q_U}\sum_{ p\leq q-2}{\lambda^{2s}_{q}\, || \tilde{u}_{p}||_{\infty} \, || \nabla h_p||_{2}  \, || h_q ||_{2}  }\\
&\lesssim \sum_{-1\leq q\leq Q_U}\sum_{ p\leq q-2}{ \lambda^{2s}_{q} \,\lambda_{p}\,|| u_{p}||_{\infty} \, || h_p||_{2} \,|| h_q ||_{2}  } \\ 
&\lesssim \sum_{-1\leq q\leq Q_U}\sum_{ p\leq q-2}{ \lambda^{2s}_{q} \,\lambda_{p}\,|| u_{\leq Q_U}||_{\infty} \, || h_p||_{2} \,|| h_q ||_{2}  } \\ 
&\lesssim c_r\, \nu \sum_{-1\leq q\leq Q_U}\sum_{  p\leq q-2}{ \lambda^{2s}_{q} \,\lambda_{p}\,\Lambda_{U} \,|| h_p||_{2} \,|| h_q ||_{2}  } \\ 
&\lesssim c_{r}\, \nu \sum_{-1\leq q\leq Q_U}{\lambda^{s+1}_{q} \, ||h_q||_{2}\, \lambda_{Q_U}} \sum_{ p\leq q}{\lambda^{s+1}_{p}\, || h_p ||_{2}\, \big( L\lambda_{q-p}\big)^{s-1} \, \lambda^{-1}_{p}  } \\
\end{split}
\end{equation}
and for $s<1$, using Jensen and Young's inequalities leads to
\begin{equation}\notag
\begin{split}
I_{531} 
&\lesssim c_{r}\, \nu\,  \sum_{-1\leq q\leq Q_U}  {\lambda^{2s+2}_{q}\,  ||h_q ||^{2}_{2}\, \lambda^2_{Q_U}} \\
&\quad+  c_{r}\, \nu\, \sum_{-1\leq q\leq Q_U} \Big( \sum_{-1\leq p\leq q} {\lambda^{s+1}_{p} \,|| h_{p} ||_{2} \, \big( L\lambda_{q-p} \big)^{s-1} \, } \Big)^{2}\,\\ 
&\lesssim c_{r}\, \nu\,  \sum_{-1\leq q\leq Q_U}  {\lambda^{2s+2}_{q}\,  ||h_q ||^{2}_{2}} +  c_{r}\, \nu\, \sum_{-1\leq q\leq Q_U} {\lambda^{2s+2}_{p} \,|| h_{p} ||^{2}_{2}} \sum_{-1\leq p\leq q}{ \big( L\lambda_{q-p} \big)^{s-1} } \,\\ 
&\lesssim c_{r}\, \nu\,  \sum_{-1\leq q\leq Q_U}  {\lambda^{2s+2}_{q}\,  ||h_q ||^{2}_{2}\,} \lesssim c_{r}\, \nu\,  \sum_{q\geq -1}  {\lambda^{2s+2}_{q}\,  ||h_q ||^{2}_{2}}
\end{split}
\end{equation}
%%%%%%%%%%%%%%%%  note that I_531  only has low modes, so we can only use low mode part of wavenumber, but not the the high mode part of wavenumber. 
where we absorbed $\lambda^2_{Q_U}$ into constants. Finally, we estimate $I_{532}$ according to,

\begin{equation}\notag
\begin{split}
I_{532}&\leq \sum_{p>Q_U}\sum_{ q\leq p+2}{\lambda^{2s}_{q}\, || \tilde{u}_{p}||_{\infty} \, || \nabla h_p||_{2}  \, || h_q ||_{2}  }\\
&\lesssim \sum_{p>Q_U}{ \lambda_{p}\,|| u_{p}||_{\infty} \, || h_p||_{2}} \sum_{ q\leq p+2}{\lambda^{2s}_{q} \,|| h_q ||_{2}  } \\ 
&\lesssim c_{r}\, \nu \sum_{p>Q_U}{\Lambda_{U}^{1+\sigma} \, \lambda^{1-\sigma}_{p}\,|| h_p||_{2}} \sum_{ Q_U<q\leq p+2}{\lambda^{2s}_{q} \, || h_q ||_{2}  }\\
%&\lesssim c_{r}\, \nu\, \sum_{p>Q_U}{\lambda^{s+1}_{p} \, ||h_p||_{2}  \,} \sum_{ Q_U<q\leq p+2}{\lambda^{s+1}_{q}\, || h_q ||_{2}\, \lambda^{s-1}_{q}\, \lambda_{Q_U}^{1+\sigma} \, \lambda^{-s-\sigma}_{p}\,  } \\ \notag
&\lesssim c_{r}\, \nu \sum_{p>Q_U}{\lambda^{s+1}_{p} \, ||h_p||_{2}} \sum_{ Q_U<q\leq p+2}{\lambda^{s+1}_{q}\, || h_q ||_{2}\, \lambda^{1+\sigma}_{Q_U-q}\, \big( L\lambda_{q-p}\big)^{s+\sigma} \,  } \\
&\lesssim c_{r}\, \nu  \sum_{p>Q_U}  {\lambda^{2s+2}_{p}\,  ||h_p ||^{2}_{2}} +  c_{r}\, \nu \sum_{p>Q_U} \Big( \sum_{Q_U<q\leq p+2} {\lambda^{s+1}_{q} \,|| h_{q} ||_{2} \, \big( L \lambda_{q-p} \big)^{s+\sigma} \, } \Big)^{2}\, \\ 
&\lesssim c_{r}\, \nu\,  \sum_{p>Q_U}  {\lambda^{2s+2}_{p}\,  ||h_p ||^{2}_{2}} \lesssim c_{r}\, \nu\,  \sum_{p\geq -1}  {\lambda^{2s+2}_{p}\,  ||h_p ||^{2}_{2}} 
\end{split}
\end{equation}
%where we assume $\sigma\geq -1$. Using Young and Jensen inequalities leads to
where we once more used $s>-\sigma$ and $\sigma\geq -1$. 

In sum, consolidating all the estimates, we obtain the estimate of $I_{5}$ as, 
\begin{equation}\label{I_5}
I_5\lesssim c_{r}\, \nu\,  \sum_{q\geq -1}  {\lambda^{2s+2}_{q}\,  ||w_q ||^{2}_{2}}  +c_{r}\, \nu\,  \sum_{q\geq -1}  {\lambda^{2s+2}_{q}\,  ||h_q ||^{2}_{2}} 
\end{equation}
\vspace{0.5cm}

\subsubsection{ Estimate for $I_{6}$}

We employ the Bony's paraproduct to decompose $I_6$ as
\begin{equation}\notag
\begin{split}
I_{6} &= -\sum_{q\geq -1}{\lambda^{2s}_{q}} \int_{\mathbb{T}^n}{\Delta_{q}(w\cdot \nabla b^{(2)})\cdot h_{q}dx}\\
&=-\sum_{q\geq -1}\sum_{\vert p-q\vert\leq 2}{\lambda^{2s}_{q}} \int_{\mathbb{T}^n}{\Delta_{q}( w_{\leq p-2}\cdot \nabla b^{(2)}_p) \cdot h_{q}dx}\\ 
&\quad- \sum_{q\geq -1}\sum_{\vert p-q\vert\leq 2}{\lambda^{2s}_{q}} \int_{\mathbb{T}^n}{\Delta_{q}( w_{p}\cdot \nabla b^{(2)}_{\leq p-2}) \cdot h_{q}dx}\\ 
&\quad- \sum_{q\geq -1}\sum_{ p\geq q-2}{\lambda^{2s}_{q}} \int_{\mathbb{T}^n}{\Delta_{q}( \tilde{w}_{p}\cdot \nabla b^{(2)}_p) \cdot h_{q}dx}\\
&=I_{61}+I_{62}+I_{63} 
\end{split}
\end{equation}
As before, since $h\vert_{\leq Q_B}=0$, we note that $I_{61}$ consists of only higher modes, 
\begin{equation}\notag
\begin{split}
|I_{61}|&\leq \sum_{q> Q_B}\sum_{\substack{\vert p-q\vert\leq 2\\ p>Q_B+2}}{\lambda^{2s}_{q}} \int_{\mathbb{T}^n}{ |\, \Delta_{q}( w_{(Q_U\, , \, p-2\, ]\, }\cdot \nabla b^{(2)}_p) \cdot h_{q} |\,dx}\\
&\leq \sum_{q>Q_B}\sum_{ \substack{\vert p-q\vert\leq 2\\ p>Q_B+2}}{\lambda^{2s}_{q}\, || \nabla b^{(2)}_{p}||_{r}\, || h_{q}||_{\frac{2r}{r-2}}\, || w_{(Q_U\, , \, p-2\, ]\, }||_{2}} \\ 
&\lesssim \sum_{q>Q_B}\sum_{ \substack{\vert p-q\vert\leq 2\\ p>Q_B+2}}{\lambda^{2s+1+\frac{n}{r}}_{q}\, || b^{(2)}_{q}||_{r}\, || h_{q}||_{2}\,  \sum_{Q_U<p^{\prime}\leq p-2}{||w_{p^{\prime}}||_{2}} } \\
&\lesssim c_{r}\, \mu \, \sum_{q>Q_B}{\lambda^{s+1}_{q}\, || h_{q}||_{2}\,  \sum_{Q_U<p^{\prime}\leq q}{\lambda^{s+1}_{p^{\prime}}\, ||w_{p^{\prime}}||_{2} \,\,  \lambda^{s-\delta}_{q} \,\Lambda^{\delta}_{B} \lambda^{-s-1}_{p^{\prime}}} }\\
&\lesssim c_{r}\, \mu \, \sum_{q>Q_B}{\lambda^{s+1}_{q}\, || h_{q}||_{2}\,  \sum_{Q_U<p^{\prime}\leq q}{\lambda^{s+1}_{p^{\prime}}\, ||w_{p^{\prime}}||_{2} \,\,  \lambda^{s-\delta}_{q-p^{\prime}} \, \lambda^{1+\delta}_{Q_{B}-p^{\prime}} \, \lambda^{-1}_{Q_B}} }
\end{split}
\end{equation}
Here we use $w_{\leq p-2}=w_{(Q_U,p-2]}+w_{\leq Q_U}=w_{(Q_U,p-2]}$, so the low-frequency truncation in the velocity factor is governed by $Q_U$, while $Q_B$ enters through $h_q$ and the bound on $b^{(2)}_p$.
considering $\delta>-1$ and using Jensen and young's inequalities, it follows that 
\begin{equation}\notag
\begin{split}
| I_{61}| &\lesssim c_{r}\, \mu \, \sum_{q>Q_B}{\lambda^{2s+2}_{q}\, || h_{q}||^{2}_{2}}\, +  c_{r}\, \mu \, \sum_{q>Q_B} \Big( \sum_{Q_U<p^{\prime}\leq q}{\lambda^{s+1}_{p^{\prime}}\, ||w_{p^{\prime}}||_{2} \,\,  \lambda^{s-\delta}_{q-p^{\prime}} } \Big)^2 \\ 
 %&\lesssim  c_{r}\, \mu \, \sum_{q>Q_B}{\lambda^{2s+2}_{q}\, || h_{q}||^{2}_{2}}\, +  c_{r}\, \mu \, \sum_{q>Q_B}{{\lambda^{2s+2}_{q}\, |||w_{q}||^{2}_{2}}}  \\ 
 &\lesssim  c_{r}\, \mu \, \sum_{q\geq -1}{\lambda^{2s+2}_{q}\, || h_{q}||^{2}_{2}}\, +  c_{r}\, \mu \, \sum_{q\geq -1}{{\lambda^{2s+2}_{q}\, ||w_{q}||^{2}_{2}}} \\ 
\end{split}
\end{equation}
where $s-\delta<0$. 

Subsequently, we proceed to estimate $I_{62}$ by splitting it into two terms using wavenumber, and estimating each one according to
\begin{equation}\notag
\begin{split}
I_{62}= &- \sum_{q >Q_B }\sum_{\vert p-q\vert\leq 2}{\lambda^{2s}_{q}} \int_{\mathbb{T}^n}{\Delta_{q}( w_{p}\cdot \nabla b^{(2)}_{\leq p-2}) \cdot h_{q}dx}\\ \notag 
= &- \sum_{q>Q_B}\sum_{\vert p-q\vert\leq 2}{\lambda^{2s}_{q}} \int_{\mathbb{T}^n}{\Delta_{q}( w_{p}\cdot \nabla b^{(2)}_{\leq Q_B}) \cdot h_{q}dx}\\ \notag
&- \sum_{q>Q_B}\sum_{\vert p-q\vert\leq 2}{\lambda^{2s}_{q}} \int_{\mathbb{T}^n}{\Delta_{q}( w_{p}\cdot \nabla b^{(2)}_{(Q_B\, , \,  p-2\, ] \, }) \cdot h_{q}dx}=I_{621}+I_{622}
\end{split}
\end{equation}
\begin{equation}\notag
\begin{split}
| I_{621}| &\leq \sum_{q>Q_B}\sum_{ \substack{\vert p-q\vert\leq 2\\ p>Q_B+2}}{\lambda^{2s}_{q}\, ||w_p ||_{2}\,  || \nabla b^{(2)}_{\leq Q_B}||_{\infty}\, || h_{q}||_{2}}\\
& \lesssim c_r\, \mu\, \sum_{q>Q_B}\sum_{\substack{\vert p-q\vert\leq 2\\ p>Q_B+2}}{\lambda^{2s}_{q}\, \Lambda_{B} \, ||w_p ||_{2}\, || h_{q}||_{2}} \\
&\lesssim  c_{r}\, \mu \,  \sum_{q>Q_B}{\lambda^{2s+2}_{q} \, ||w_q ||_{2}\,|| h_{q}||_{2} \,  \lambda_{Q_B}\lambda^{-2}_{q} } \lesssim \sum_{q>Q_B}{\lambda^{2s+2}_{q}\, ||w_q ||_{2}\, || h_{q}||_{2} \, } \\ 
%&\lesssim c_{r}\, \mu \, \sum_{q>Q_B}{ \lambda^{2s+2}_{q}\,  || w_{q}||^{2}_{2}} + c_{r}\, \mu \,\sum_{q>Q_B} { \lambda^{2s+2}_{q}\, ||h_{q} ||^{2}_{2}} \\ 
&\lesssim c_{r}\, \mu \, \sum_{q\geq -1}{ \lambda^{2s+2}_{q}\,  || w_{q}||^{2}_{2}} + c_{r}\, \mu \,\sum_{q\geq -1} { \lambda^{2s+2}_{q}\, ||h_{q} ||^{2}_{2}}
\end{split}
\end{equation}
where we used the fact that $\lambda_{Q_B}\lambda^{-2}_{q}<1$ for $q>Q_B$. Similarly, we have
\begin{equation}\notag
\begin{split}
    | I_{622}| &\leq \sum_{q>Q_B}\sum_{ \substack{\vert p-q\vert\leq 2\\ \, p>Q_B+2}}{\lambda^{2s}_{q}\, ||h_q ||_{2}\, ||w_p ||_{\frac{2r}{r-2}}\,   || \nabla b^{(2)}_{(\, Q_B\, ,\, p-2\,]\, }||_{r}\, } \\ 
  & \lesssim \sum_{q>Q_B}\sum_{ \substack{\vert p-q\vert\leq 2\\ \, p>Q_B+2}}{\lambda^{2s}_{q}\, \, ||h_q ||_{2}\, \lambda^{\frac{n}{r}}_{p}\,  ||w_p ||_{2} }\, \sum_{Q_B<p^{\prime}\leq p-2}{ || \nabla b^{(2)}_{p^{\prime}}||_{r}\, } \\ 
   & \lesssim \sum_{q>Q_B}{\lambda^{2s+2}_{q}\, \, ||h_q ||_{2}\,   ||w_q ||_{2} }\, \sum_{Q_B<p^{\prime}\leq q}{  \lambda^{\frac{n}{r}-2}_{q}\, \lambda_{p^{\prime}}\, || b^{(2)}_{p^{\prime}}||_{r}\, } \\ 
    &\lesssim c_{r}\, \mu \,  \sum_{q>Q_B}{\lambda^{2s+2}_{q}\, \, ||h_q ||_{2}\,   ||w_q ||_{2} }\, \sum_{Q_B<p^{\prime}\leq q}{ \,\lambda^{\frac{n}{r}-2}_{q}\, \lambda^{1-\frac{n}{r}-\delta}_{p^{\prime}} \, \Lambda^{\delta}_B  } \\ 
    &\lesssim c_{r}\, \mu \,  \sum_{q>Q_B}{\lambda^{2s+2}_{q}\, \, ||h_q ||_{2}\,   ||w_q ||_{2} }\, \sum_{Q_B<p^{\prime}\leq q}{  (L\lambda_{q-p^{\prime}})^{\frac{n}{r}-2}\, \, \lambda^{1+\delta}_{Q_{B}-p^{\prime}} \,\lambda_{Q_B}^{-1} } \\ 
\end{split}
\end{equation}
where $\delta>-1$ as before, and we also assume $r>\frac{n}{2}$. Therefore, we obtain
\begin{equation}\notag
\begin{split}
    | I_{622}| & \lesssim c_{r}\, \mu \,  \sum_{q>Q_B}{\lambda^{2s+2}_{q}\,  ||h_q ||^{2}_{2} }+c_{r}\, \mu \,  \sum_{q>Q_B}{\lambda^{2s+2}_{q}\,  ||w_q ||^{2}_{2} }\, \\ 
    &\lesssim c_{r}\, \mu \,  \sum_{q\geq -1}{\lambda^{2s+2}_{q}\,  ||h_q ||^{2}_{2} }+c_{r}\, \mu \,  \sum_{q\geq -1}{\lambda^{2s+2}_{q}\,  ||w_q ||^{2}_{2} }\, \\ 
\end{split}
\end{equation}
Subsequently, we proceed with estimation of $I_{63}$. Analogous to $I_{33}$, we make use of wavenumber to split $I_{63}$ into 
\begin{equation}\notag
\begin{split}
   I_{63} =& - \sum_{q\geq -1}\sum_{ p\geq q-2}{\lambda^{2s}_{q}} \int_{\mathbb{T}^n}{\Delta_{q}( \tilde{w}_{p}\cdot \nabla b^{(2)}_p) \cdot h_{q}dx}  \\
   =&- \sum_{Q_B-1\leq p\leq Q_B }\sum_{ Q_B<q\leq p+2}{\lambda^{2s}_{q}} \int_{\mathbb{T}^n}{\Delta_{q}( \tilde{w}_{p}\cdot \nabla b^{(2)}_p) \cdot h_{q}dx}   \\ 
   &- \sum_{p>Q_B}\sum_{ q\leq p+2}{\lambda^{2s}_{q}} \int_{\mathbb{T}^n}{\Delta_{q}( \tilde{w}_{p}\cdot \nabla b^{(2)}_p) \cdot h_{q}dx} =I_{631}+I_{632} 
\end{split}
\end{equation}
where we used the fact that $h\vert_{\leq Q_B}=0$, so $q>Q_B$, and hence
the conditions $p\geq q-2$ and $p\leq Q_B$ force $Q_B-1\leq p\leq Q_B$.
We estimate each term separately.
\begin{equation}\notag
\begin{split}
    |I_{631}|& \leq \sum_{Q_B-1\leq p\leq Q_B }\sum_{ Q_B<q\leq p+2}{\lambda^{2s}_{q}\, || \nabla b^{(2)}_{p}||_{\infty}\, ||\tilde{w}_{p} ||_{2}\, ||h_q ||_{2}\,    } \\
       &\lesssim c_{r}\, \mu \sum_{Q_B-1\leq p\leq Q_B }\sum_{ Q_B<q\leq p+2}{\lambda^{2s}_{q}\,\Lambda_B\, ||\tilde{w}_{p} ||_{2}\, ||h_q ||_{2}\,    } \\
    & \lesssim c_{r}\, \mu \sum_{Q_B<q\leq Q_B+2}{\lambda^{2s+2}_{q}\,  ||h_q ||^{2}_{2} }+c_{r}\, \mu \sum_{Q_B-2\leq p\leq Q_B+1}{\lambda^{2s+2}_{p}\,  ||w_p ||^{2}_{2} } \\
    & \lesssim c_{r}\, \mu \sum_{q\geq -1}{\lambda^{2s+2}_{q}\,  ||w_q ||^{2}_{2} }+c_{r}\, \mu \sum_{q\geq -1}{\lambda^{2s+2}_{q}\,  ||h_q ||^{2}_{2} }.
   \end{split}
\end{equation}
where we used $\|\nabla b^{(2)}_p\|_{\infty}\leq \|\nabla b^{(2)}_{\leq Q_B}\|_{\infty}\lesssim c_r\mu \Lambda_B$ and $\|\tilde{w}_p\|_2^2\lesssim \sum_{|j|\leq 1}\|w_{p+j}\|_2^2$; since only finitely many indices occur and $|q-p|\leq 2$, the dyadic factors are comparable and can be absorbed into constants. 
Finally we estimate $I_{632}$ according to

\begin{equation}\notag
\begin{split}
    |I_{632}|& \leq \sum_{p>Q_B}\sum_{ q\leq p+2}{\lambda^{2s}_{q}\, || \nabla b^{(2)}_{p}||_{r}\, ||\tilde{w}_{p} ||_{\frac{2r}{r-2}}\, ||h_q ||_{2}\, }\\
    & \lesssim \sum_{p>Q_B}{ \lambda_{p} \, || b^{(2)}_{p}||_{r}\,  \lambda^{\frac{n}{r}}_{p}\, ||w_{p} ||_{2}\,} \sum_{ q\leq p+2}{ \lambda^{2s}_{q}\, ||h_q ||_{2}\, }\\
    &\lesssim c_r \mu \sum_{p>Q_B}{ \lambda^{s+1}_{p}\, ||w_{p} ||_{2}\, } \sum_{ Q_B<q\leq p+2}{ \lambda^{s+1}_{q}\, ||h_q ||_{2}\,\, (L\lambda_{q-p})^{s+\delta}\, (\lambda_{Q_{B}-q})^{1+\delta}\,\, \lambda_{Q_B}^{-1}} 
 \end{split}
\end{equation}
where $\delta>-1$. Also letting $s>-\delta$ and using Jensen and Young's inequalities gives
\begin{equation}\notag
\begin{split}
    |I_{632}| &\lesssim c_{r}\, \mu \,  \sum_{p>Q_B}{ \,  \lambda^{2s+2}_{p}\, ||w_{p} ||^{2}_{2}\,} +c_{r}\, \mu \,   \sum_{p>Q_B} \Big( \sum_{ Q_B<q\leq p+2}{ \lambda^{s+1}_{q}\, ||h_q ||_{2}\,\, (L\lambda_{q-p})^{s+\delta}\, } \Big)^2 \\ 
    &\lesssim c_{r}\, \mu \,  \sum_{p\geq -1}{ \,  \lambda^{2s+2}_{p}\, ||w_{p} ||^{2}_{2}\,} + c_{r}\, \mu \,  \sum_{p\geq -1}{ \,  \lambda^{2s+2}_{p}\, ||h_{p} ||^{2}_{2}\,}
  \end{split}
\end{equation}

Taken together all the estimates, we obtain the estimate of $I_{6}$ according to, 

\begin{equation} \label{I_6}
I_6 \lesssim c_{r}\, \mu\,  \sum_{q\geq -1}  {\lambda^{2s+2}_{q}\,  ||w_q ||^{2}_{2}}  + c_{r}\, \mu\,  \sum_{q\geq -1}  {\lambda^{2s+2}_{q}\,  ||h_q||^{2}_{2}} 
\end{equation}

\vspace{0.5cm}

\subsubsection{Estimate for $I_7$}

The term $I_7$ is handled by the same Bony decomposition as in the proof of
\eqref{I_3}. Writing $I_7=I_{71}+I_{72}+I_{73}$, the terms $I_{71}$ and
$I_{73}$ are estimated exactly as the corresponding pieces $I_{31}$ and
$I_{33}$, with $b^{(1)}$ replaced by $u$ and the velocity bounds
\eqref{Hall-high-modes} and \eqref{Hall-low-modes} used in place of the
magnetic wavenumber bounds. In particular, the $h$-factor is still governed by
$Q_B$: since $h_{\leq Q_B}=0$ by \eqref{Hall-low-vanishing}, the low-frequency
magnetic contribution vanishes in the same way as in the proof of \eqref{I_3}.
For $I_{72}$ one instead splits
$u_{\leq p-2}=u_{(Q_U,p-2]}+u_{\leq Q_U}$ and repeats the argument used for
the estimate of $I_{62}$, so here $Q_U$ enters through the velocity factor
while $Q_B$ still appears through the restriction on $h$. With these two
cutoffs playing their respective roles, the same convolution estimates yield
\begin{equation}\label{I_7}
|I_7|\lesssim c_r \nu \sum_{q\geq -1}\lambda_q^{2s+2}\|h_q\|_2^2
\end{equation}
for $-\sigma<s<1$. Since $r>n$, we have $\frac{n}{r}-1<1$, so this condition
is compatible with the range in Theorem \ref{thm-Hall-det}.

\vspace{0.5cm}

\subsubsection{Estimate for $I_8$}

The estimate of $I_8$ follows the same Bony decomposition and convolution
estimates as in the proof of \eqref{I_4}. In the middle paraproduct, for
$p>Q_U+2$, we write
\[
w_{\leq p-2}=w_{(Q_U,p-2]}+w_{\leq Q_U}=w_{(Q_U,p-2]},
\]
because $w_{\leq Q_U}=0$ by \eqref{Hall-low-vanishing}. The terms involving
$b^{(2)}_{\leq Q_B}$ are then controlled by the low-mode bound in the
definition of $\Lambda_B$, while the terms with $b^{(2)}_{(Q_B,p-2]}$ and
$\widetilde{b}^{(2)}_p$ are handled by \eqref{Hall-high-modes}, exactly as in
\eqref{I_4}. Repeating those estimates gives
\begin{equation}\label{I_8}
|I_8|\lesssim c_r \mu \sum_{q\geq -1}\lambda_q^{2s+2}\|w_q\|_2^2
+c_r \mu \sum_{q\geq -1}\lambda_q^{2s+2}\|h_q\|_2^2
\end{equation}
for $-\delta<s<1+\delta$. Under the assumptions of Theorem
\ref{thm-Hall-det}, this range is satisfied.

\vspace{0.5cm}

To finish the proof, we combine the standard velocity estimate from Ref.
\cite{1},
\[
|I_1|+|I_2|\lesssim c_r \nu \sum_{q\geq -1}\lambda_q^{2s+2}\|w_q\|_2^2,
\]
the Hall estimates proved in Subsection 4.1,
\[
|J|+|K|\lesssim c_r \mu \sum_{q\geq -1}\lambda_q^{2s+2}\|h_q\|_2^2,
\]
and \eqref{I_3}--\eqref{I_8}. Hence, after choosing $c_r>0$ sufficiently
small in the definitions of $\Lambda_U$ and $\Lambda_B$, there exists a
constant $c_0>0$ such that
\begin{equation}\label{Hall-final-energy}
\frac{d}{dt}\,X(t)+c_0\sum_{q\geq -1}\lambda_q^{2s+2}
\Big(\nu\|w_q\|_2^2+\mu\|h_q\|_2^2\Big)\leq 0,
\end{equation}
where
\[
X(t):=\sum_{q\geq -1}\lambda_q^{2s}\Big(\|w_q\|_2^2+\|h_q\|_2^2\Big).
\]
Since $\lambda_q\geq \lambda_{-1}>0$ on $\mathbb{T}^n$, the dissipative term
controls $X(t)$:
\[
\sum_{q\geq -1}\lambda_q^{2s+2}\Big(\nu\|w_q\|_2^2+\mu\|h_q\|_2^2\Big)
\geq c_1 X(t)
\]
for some $c_1>0$ depending only on $\nu$, $\mu$, and $L$. Therefore
\[
\frac{d}{dt}X(t)+c_1X(t)\leq 0,
\]
and Gronwall's inequality implies
\[
X(t)\leq X(0)e^{-c_1 t}.
\]
Consequently,
\[
\lim_{t\to\infty}\Big(\|u(t)-v(t)\|_{H^s}
+\|b^{(1)}(t)-b^{(2)}(t)\|_{H^s}\Big)=0.
\]
This completes the proof of Theorem \ref{thm-Hall-det}.

\section{Average Determining Wavenumbers and Dissipation Scales}

\subsection{EMHD}

For the three-dimensional EMHD model, the phenomenology in \cite{3} predicts
the magnetic dissipation number
\begin{equation}\label{eq:kappa-emhd}
\kappa_e^{\delta_b}
:=
\left(\frac{\varepsilon_b}{\mu^3}\right)^{\frac{1}{\delta_b-1}},
\qquad
\varepsilon_b
:=
\mu \lambda_0^{\delta_b}
\left\langle
\sum_{q\ge -1}\lambda_q^2\|b_q\|_2^2
\right\rangle,
\end{equation}
where \(\delta_b\in(0,3)\) is an intermittency dimension and
\begin{equation}\label{eq:time-average}
\langle f\rangle := \frac{1}{T}\int_t^{t+T} f(\tau)\, d\tau.
\end{equation}
To compare \(\kappa_e^{\delta_b}\) with the determining wavenumber
\(\Lambda_{b,r}\), we assume the scale-localized intermittency relation
\begin{equation}\label{eq:intermdef-emhd}
\left\langle
\sum_{q\ge -1}\lambda_q^{-1+\delta_b+\frac{6}{r}}\|b_q\|_r^2
\right\rangle
\lesssim
\lambda_0^{\delta_b}
\left\langle
\sum_{q\ge -1}\lambda_q^{2}\|b_q\|_2^2
\right\rangle.
\end{equation}
This is the only intermittency input used below.

\begin{Lemma}\label{lem:emhd-average-pointwise}
Assume \(0<\delta<1\) and \(2\delta+1<\delta_b<3\). Let \(Q=Q(t)\) be such
that \(\Lambda_{b,r}(t)=\lambda_Q\). Then for every \(t\),
\begin{equation}\label{eq:emhd-pointwise-average}
(c_r\mu)^2\bigl(\Lambda_{b,r}(t)-\lambda_0\bigr)_+^{\delta_b-1}
\lesssim
\sum_{q\ge -1}\lambda_q^{-1+\delta_b+\frac{6}{r}}\|b_q(t)\|_r^2.
\end{equation}
\end{Lemma}

\begin{proof}
If \(\Lambda_{b,r}(t)=\lambda_0\), there is nothing to prove. Assume first that
\(\lambda_0<\Lambda_{b,r}(t)<\infty\). Since \(q=Q-1\) does not belong to the
admissible set in \eqref{wave}, at least one of the two defining conditions
fails at level \(Q-1\).

If the high-frequency condition fails, then for some \(p\ge Q\),
\[
(L\lambda_{p-Q+1})^\delta \lambda_p^{3/r}\|b_p\|_r\ge c_r\mu.
\]
Using \(\Lambda_{b,r}=\lambda_Q\) and \(L\lambda_{p-Q+1}=2L\lambda_{p-Q}\), we obtain
\[
(c_r\mu)^2 \Lambda_{b,r}^{\delta_b-1}
\lesssim
(L\lambda_{p-Q})^{2\delta}\Lambda_{b,r}^{\delta_b-1}\lambda_p^{6/r}\|b_p\|_r^2.
\]
Since
\[
\Lambda_{b,r}^{\delta_b-1}
=
\lambda_p^{\delta_b-1}(L\lambda_{p-Q})^{1-\delta_b},
\]
it follows that
\[
(c_r\mu)^2 \Lambda_{b,r}^{\delta_b-1}
\lesssim
(L\lambda_{p-Q})^{2\delta+1-\delta_b}
\lambda_p^{-1+\delta_b+\frac{6}{r}}\|b_p\|_r^2.
\]
Because \(\delta_b>2\delta+1\), the exponent \(2\delta+1-\delta_b\) is
negative, hence
\[
(c_r\mu)^2 \Lambda_{b,r}^{\delta_b-1}
\lesssim
\lambda_p^{-1+\delta_b+\frac{6}{r}}\|b_p\|_r^2
\le
\sum_{q\ge -1}\lambda_q^{-1+\delta_b+\frac{6}{r}}\|b_q\|_r^2.
\]

If instead the low-frequency condition fails, then
\[
\lambda_{Q-1}^{-1}\|\nabla b_{\le Q-1}\|_\infty \ge c_r\mu,
\]
so \(\|\nabla b_{\le Q-1}\|_\infty \gtrsim c_r\mu \Lambda_{b,r}\). Therefore
\[
(c_r\mu)^2 \Lambda_{b,r}^{\delta_b-1}
\lesssim
\Lambda_{b,r}^{\delta_b-3}\|\nabla b_{\le Q-1}\|_\infty^2.
\]
By Bernstein and Cauchy-Schwarz,
\begin{align*}
\Lambda_{b,r}^{\delta_b-3}\|\nabla b_{\le Q-1}\|_\infty^2
&\lesssim
\Lambda_{b,r}^{\delta_b-3}
\Big(\sum_{q<Q}\lambda_q^{1+\frac{3}{r}}\|b_q\|_r\Big)^2 \\
&=
\Lambda_{b,r}^{\delta_b-3}
\Big(
\sum_{q<Q}
\lambda_q^{\frac{-1+\delta_b+6/r}{2}}\|b_q\|_r
(L\lambda_{Q-q})^{\frac{\delta_b-3}{2}}
\Big)^2 \\
&\lesssim
\Big(
\sum_{q<Q}\lambda_q^{-1+\delta_b+\frac{6}{r}}\|b_q\|_r^2
\Big)
\Big(
\sum_{q<Q}(L\lambda_{Q-q})^{\delta_b-3}
\Big).
\end{align*}
Since \(\delta_b<3\), the last geometric sum is bounded independently of \(Q\).
This proves \eqref{eq:emhd-pointwise-average} when
\(\Lambda_{b,r}(t)<\infty\).

It remains to treat the case \(\Lambda_{b,r}(t)=\infty\). Let
\[
S(t):=\sum_{q\ge -1}\lambda_q^{-1+\delta_b+\frac{6}{r}}\|b_q(t)\|_r^2.
\]
We claim that \(S(t)=\infty\). Indeed, if \(S(t)<\infty\), then for every
\(p>q\),
\begin{align*}
(L\lambda_{p-q})^\delta \lambda_p^{3/r}\|b_p\|_r
&=
\lambda_q^{\frac{1-\delta_b}{2}}
(L\lambda_{p-q})^{\delta+\frac{1-\delta_b}{2}}
\Big(
\lambda_p^{-1+\delta_b+\frac{6}{r}}\|b_p\|_r^2
\Big)^{1/2} \\
&\le
\lambda_q^{\frac{1-\delta_b}{2}}
(L\lambda_{p-q})^{\delta+\frac{1-\delta_b}{2}} S(t)^{1/2}.
\end{align*}
Because \(\delta_b>2\delta+1\), the exponent
\(\delta+\frac{1-\delta_b}{2}\) is negative, so the right-hand side is bounded
by \(C\lambda_q^{\frac{1-\delta_b}{2}}S(t)^{1/2}\), which tends to \(0\) as
\(q\to\infty\). Hence the high-frequency condition in \eqref{wave} holds for
all sufficiently large \(q\).

Similarly,
\begin{align*}
\lambda_q^{-1}\|\nabla b_{\le q}\|_\infty
&\lesssim
\lambda_q^{-1}\sum_{p\le q}\lambda_p^{1+\frac{3}{r}}\|b_p\|_r \\
&=
\lambda_q^{\frac{1-\delta_b}{2}}
\sum_{p\le q}
(L\lambda_{q-p})^{\frac{\delta_b-3}{2}}
\Big(
\lambda_p^{-1+\delta_b+\frac{6}{r}}\|b_p\|_r^2
\Big)^{1/2} \\
&\lesssim
\lambda_q^{\frac{1-\delta_b}{2}}S(t)^{1/2},
\end{align*}
where we used \(\delta_b<3\) in the last step. The right-hand side again tends
to \(0\), so the low-frequency condition in \eqref{wave} also holds for all
sufficiently large \(q\). This contradicts \(\Lambda_{b,r}(t)=\infty\). Hence
\(S(t)=\infty\), and \eqref{eq:emhd-pointwise-average} remains valid with the
convention \(\infty\le \infty\).
\end{proof}

\begin{Theorem}\label{thm:avg-emhd}
Assume \(0<\delta<1\) and \(\max\{2,2\delta+1\}<\delta_b<3\). If the
intermittency relation \eqref{eq:intermdef-emhd} holds, then
\begin{equation}\label{eq:avg-emhd}
\langle \Lambda_{b,r}\rangle \lesssim \lambda_0 + \kappa_e^{\delta_b}.
\end{equation}
\end{Theorem}

\begin{proof}
Since \(\delta_b-1>1\), the monotonicity of \(L^p\) norms gives
\[
\left\langle (\Lambda_{b,r}-\lambda_0)_+ \right\rangle
\le
\left\langle (\Lambda_{b,r}-\lambda_0)_+^{\delta_b-1} \right\rangle^{\frac{1}{\delta_b-1}}.
\]
Applying Lemma \ref{lem:emhd-average-pointwise} and then
\eqref{eq:intermdef-emhd}, we obtain
\begin{align*}
\left\langle (\Lambda_{b,r}-\lambda_0)_+ \right\rangle
&\lesssim
\left[
\frac{1}{\mu^2}
\left\langle
\sum_{q\ge -1}\lambda_q^{-1+\delta_b+\frac{6}{r}}\|b_q\|_r^2
\right\rangle
\right]^{\frac{1}{\delta_b-1}} \\
&\lesssim
\left[
\frac{\lambda_0^{\delta_b}}{\mu^2}
\left\langle
\sum_{q\ge -1}\lambda_q^2\|b_q\|_2^2
\right\rangle
\right]^{\frac{1}{\delta_b-1}} \\
&=
\left(\frac{\varepsilon_b}{\mu^3}\right)^{\frac{1}{\delta_b-1}}
=
\kappa_e^{\delta_b}.
\end{align*}
Therefore
\[
\langle \Lambda_{b,r}\rangle
\le
\lambda_0 + \left\langle (\Lambda_{b,r}-\lambda_0)_+ \right\rangle
\lesssim
\lambda_0 + \kappa_e^{\delta_b}.
\]
\end{proof}

\subsection{Hall-MHD}

For Hall-MHD we compare the velocity and magnetic determining wavenumbers with
the dissipation numbers suggested by the Navier-Stokes and EMHD scalings. For a
sufficiently regular Hall-MHD solution \((u,b)\), define
\begin{equation}\label{eq:kappa-hall}
\kappa_u^{\delta_u}
:=
\left(\frac{\varepsilon_u}{\nu^3}\right)^{\frac{1}{1+\delta_u}},
\qquad
\varepsilon_u
:=
\nu \lambda_0^{\delta_u}
\left\langle
\sum_{q\ge -1}\lambda_q^2\|u_q\|_2^2
\right\rangle,
\end{equation}
and
\begin{equation}\label{eq:kappa-hall-b}
\kappa_e^{\delta_b}
:=
\left(\frac{\varepsilon_b}{\mu^3}\right)^{\frac{1}{\delta_b-1}},
\qquad
\varepsilon_b
:=
\mu \lambda_0^{\delta_b}
\left\langle
\sum_{q\ge -1}\lambda_q^2\|b_q\|_2^2
\right\rangle.
\end{equation}
We assume the intermittency relations
\begin{equation}\label{eq:hall-int-u}
\left\langle
\sum_{q\ge -1}\lambda_q^{-1+\delta_u}\|u_q\|_\infty^2
\right\rangle
\lesssim
\lambda_0^{\delta_u}
\left\langle
\sum_{q\ge -1}\lambda_q^2\|u_q\|_2^2
\right\rangle
\end{equation}
and
\begin{equation}\label{eq:hall-int-b}
\left\langle
\sum_{q\ge -1}\lambda_q^{-1+\delta_b+\frac{6}{r}}\|b_q\|_r^2
\right\rangle
\lesssim
\lambda_0^{\delta_b}
\left\langle
\sum_{q\ge -1}\lambda_q^2\|b_q\|_2^2
\right\rangle.
\end{equation}
The second condition is the magnetic analogue of
\eqref{eq:intermdef-emhd}; the first is the standard intermittency relation
associated with the Navier-Stokes dissipation number.

\begin{Lemma}\label{lem:hall-u-pointwise}
Assume \(0<\sigma<1\) and \(2\sigma+1<\delta_u<3\). Let \(Q_u=Q_u(t)\) be such
that \(\Lambda_{u,r}(t)=\lambda_{Q_u}\). Then for every \(t\),
\begin{equation}\label{eq:hall-u-pointwise}
(c_r\nu)^2\bigl(\Lambda_{u,r}(t)-\lambda_0\bigr)_+^{1+\delta_u}
\lesssim
\sum_{q\ge -1}\lambda_q^{-1+\delta_u}\|u_q(t)\|_\infty^2.
\end{equation}
\end{Lemma}

\begin{proof}
If \(\Lambda_{u,r}(t)=\lambda_0\), the claim is trivial. Assume first that
\(\lambda_0<\Lambda_{u,r}(t)<\infty\). Since \(q=Q_u-1\) is not admissible in
the definition of \(\Lambda_{u,r}\), one of the two defining conditions fails
there.

If the high-frequency condition fails, then for some \(p\ge Q_u\),
\[
(L\lambda_{p-Q_u+1})^\sigma \lambda_{Q_u-1}^{-1}\|u_p\|_\infty \ge c_r\nu.
\]
Since \(\lambda_{Q_u-1}=\Lambda_{u,r}/2\), we obtain
\[
(c_r\nu)^2 \Lambda_{u,r}^{1+\delta_u}
\lesssim
(L\lambda_{p-Q_u})^{2\sigma}\Lambda_{u,r}^{\delta_u-1}\|u_p\|_\infty^2.
\]
Using
\[
\Lambda_{u,r}^{\delta_u-1}
=
\lambda_p^{\delta_u-1}(L\lambda_{p-Q_u})^{1-\delta_u},
\]
we get
\[
(c_r\nu)^2 \Lambda_{u,r}^{1+\delta_u}
\lesssim
(L\lambda_{p-Q_u})^{2\sigma+1-\delta_u}
\lambda_p^{-1+\delta_u}\|u_p\|_\infty^2.
\]
Because \(\delta_u>2\sigma+1\), the exponent \(2\sigma+1-\delta_u\) is
negative, hence
\[
(c_r\nu)^2 \Lambda_{u,r}^{1+\delta_u}
\lesssim
\sum_{q\ge -1}\lambda_q^{-1+\delta_u}\|u_q\|_\infty^2.
\]

If the low-frequency condition fails, then
\[
\lambda_{Q_u-1}^{-2}\|\nabla u_{\le Q_u-1}\|_\infty \ge c_r\nu,
\]
so \(\|\nabla u_{\le Q_u-1}\|_\infty \gtrsim c_r\nu \Lambda_{u,r}^2\).
Therefore
\[
(c_r\nu)^2 \Lambda_{u,r}^{1+\delta_u}
\lesssim
\Lambda_{u,r}^{\delta_u-3}\|\nabla u_{\le Q_u-1}\|_\infty^2.
\]
By Bernstein and Cauchy-Schwarz,
\begin{align*}
\Lambda_{u,r}^{\delta_u-3}\|\nabla u_{\le Q_u-1}\|_\infty^2
&\lesssim
\Lambda_{u,r}^{\delta_u-3}
\Big(\sum_{q<Q_u}\lambda_q\|u_q\|_\infty\Big)^2 \\
&=
\Lambda_{u,r}^{\delta_u-3}
\Big(
\sum_{q<Q_u}
\lambda_q^{\frac{-1+\delta_u}{2}}\|u_q\|_\infty
(L\lambda_{Q_u-q})^{\frac{\delta_u-3}{2}}
\Big)^2 \\
&\lesssim
\Big(
\sum_{q<Q_u}\lambda_q^{-1+\delta_u}\|u_q\|_\infty^2
\Big)
\Big(
\sum_{q<Q_u}(L\lambda_{Q_u-q})^{\delta_u-3}
\Big).
\end{align*}
Since \(\delta_u<3\), the geometric sum is bounded independently of \(Q_u\).
This yields \eqref{eq:hall-u-pointwise} when \(\Lambda_{u,r}(t)<\infty\).

Finally, let
\[
S_u(t):=\sum_{q\ge -1}\lambda_q^{-1+\delta_u}\|u_q(t)\|_\infty^2.
\]
Assume \(S_u(t)<\infty\). Then for \(p>q\),
\begin{align*}
(L\lambda_{p-q})^\sigma \lambda_q^{-1}\|u_p\|_\infty
&=
\lambda_q^{-\frac{1+\delta_u}{2}}
(L\lambda_{p-q})^{\sigma+\frac{1-\delta_u}{2}}
\Big(
\lambda_p^{-1+\delta_u}\|u_p\|_\infty^2
\Big)^{1/2} \\
&\le
\lambda_q^{-\frac{1+\delta_u}{2}}
(L\lambda_{p-q})^{\sigma+\frac{1-\delta_u}{2}}
S_u(t)^{1/2}.
\end{align*}
Because \(\delta_u>2\sigma+1\), the exponent
\(\sigma+\frac{1-\delta_u}{2}\) is negative, so the right-hand side is bounded
by \(C\lambda_q^{-(1+\delta_u)/2}S_u(t)^{1/2}\), which tends to \(0\) as
\(q\to\infty\). Hence the high-frequency condition holds for all sufficiently
large \(q\).

Likewise,
\begin{align*}
\lambda_q^{-2}\|\nabla u_{\le q}\|_\infty
&\lesssim
\lambda_q^{-2}\sum_{p\le q}\lambda_p\|u_p\|_\infty \\
&=
\lambda_q^{-\frac{1+\delta_u}{2}}
\sum_{p\le q}
(L\lambda_{q-p})^{\frac{\delta_u-3}{2}}
\Big(
\lambda_p^{-1+\delta_u}\|u_p\|_\infty^2
\Big)^{1/2} \\
&\lesssim
\lambda_q^{-\frac{1+\delta_u}{2}} S_u(t)^{1/2},
\end{align*}
where we used \(\delta_u<3\). The right-hand side again tends to \(0\), so the
low-frequency condition also holds for all sufficiently large \(q\). This
contradicts \(\Lambda_{u,r}(t)=\infty\). Therefore \(S_u(t)=\infty\) whenever
\(\Lambda_{u,r}(t)=\infty\), and \eqref{eq:hall-u-pointwise} remains valid with
the convention \(\infty\le \infty\).
\end{proof}

\begin{Theorem}\label{thm:avg-hall}
Assume \(0<\sigma<1\), \(0<\delta<1\),
\[
2\sigma+1<\delta_u<3,
\qquad
\max\{2,2\delta+1\}<\delta_b<3.
\]
If the intermittency relations \eqref{eq:hall-int-u} and
\eqref{eq:hall-int-b} hold, then
\begin{equation}\label{eq:avg-hall}
\langle \Lambda_{u,r}\rangle \lesssim \lambda_0 + \kappa_u^{\delta_u},
\qquad
\langle \Lambda_{b,r}\rangle \lesssim \lambda_0 + \kappa_e^{\delta_b}.
\end{equation}
\end{Theorem}

\begin{proof}
Since \(1+\delta_u>1\), Lemma \ref{lem:hall-u-pointwise} gives
\begin{align*}
\left\langle (\Lambda_{u,r}-\lambda_0)_+ \right\rangle
&\le
\left\langle (\Lambda_{u,r}-\lambda_0)_+^{1+\delta_u} \right\rangle^{\frac{1}{1+\delta_u}} \\
&\lesssim
\left[
\frac{1}{\nu^2}
\left\langle
\sum_{q\ge -1}\lambda_q^{-1+\delta_u}\|u_q\|_\infty^2
\right\rangle
\right]^{\frac{1}{1+\delta_u}} \\
&\lesssim
\left[
\frac{\lambda_0^{\delta_u}}{\nu^2}
\left\langle
\sum_{q\ge -1}\lambda_q^2\|u_q\|_2^2
\right\rangle
\right]^{\frac{1}{1+\delta_u}} \\
&=
\left(\frac{\varepsilon_u}{\nu^3}\right)^{\frac{1}{1+\delta_u}}
=
\kappa_u^{\delta_u}.
\end{align*}
Hence \(\langle \Lambda_{u,r}\rangle \lesssim \lambda_0 + \kappa_u^{\delta_u}\).

For the magnetic field, the proof of Theorem \ref{thm:avg-emhd} uses only the
definition of \(\Lambda_{b,r}\) and the intermittency relation, not the EMHD
evolution itself. Applying the same argument to the Hall magnetic field \(b\)
and using \eqref{eq:hall-int-b} yields
\[
\langle \Lambda_{b,r}\rangle \lesssim \lambda_0 + \kappa_e^{\delta_b}.
\]
This proves \eqref{eq:avg-hall}.
\end{proof}

\section{Uniform Bounds from Besov Regularity}\label{sec-uniform-bound}

We conclude with a direct criterion guaranteeing that the determining
wavenumber stays uniformly bounded. The factor \((L\lambda_{p-q})^\delta\) in
\eqref{wave} forces \(\delta\) derivatives of high-frequency control, so the
natural assumption here is a Besov bound at the level
\(B^{\delta+n/r}_{r,\infty}\).

\begin{Theorem}\label{thm:uniform-bound-Lambda-b}
Let \(n\geq 2\), let \(1\leq r\leq \infty\), and let \(\delta>1\). Suppose that
\(b=b(x,t)\) is a strong solution of \eqref{emhd}, and let
\(\Lambda_{b,r}(t)\) be the magnetic determining wavenumber defined in
\eqref{wave}. Assume that
\begin{equation}\notag
\sup_{t\geq 0}\|b(t)\|_{B^{\delta+n/r}_{r,\infty}}<\infty.
\end{equation}
Then there exists a constant \(C=C(r,\delta,n,L,c_r)\) such that for every
\(t\geq 0\),
\begin{equation}\notag
\Lambda_{b,r}(t)
\leq
\frac{C}{\mu}\,\|b(t)\|_{B^{\delta+n/r}_{r,\infty}}.
\end{equation}
In particular, if
\begin{equation}\notag
M_\delta:=\sup_{t\geq 0}\|b(t)\|_{B^{\delta+n/r}_{r,\infty}}<\infty,
\end{equation}
then
\begin{equation}\notag
\sup_{t\geq 0}\Lambda_{b,r}(t)\leq \frac{C}{\mu}M_\delta<\infty.
\end{equation}
\end{Theorem}

\begin{proof}
Fix \(t\geq 0\), and write
\[
M_\delta(t):=\|b(t)\|_{B^{\delta+n/r}_{r,\infty}}
=
\sup_{p\geq -1}\lambda_p^{\delta+n/r}\|b_p(t)\|_{L^r}.
\]
We show that every dyadic frequency \(\lambda_q\) satisfying
\[
\lambda_q\geq \frac{C}{\mu}M_\delta(t)
\]
for a sufficiently large constant \(C\) belongs to the admissible set in
\eqref{wave}. This implies the desired estimate by minimality.

\medskip

\noindent
\emph{Step 1: verification of the high-frequency condition.}
Let \(p>q\). By the definition of \(M_\delta(t)\),
\[
\|b_p(t)\|_{L^r}\leq M_\delta(t)\lambda_p^{-\delta-n/r}.
\]
Therefore
\[
(L\lambda_{p-q})^\delta \lambda_p^{n/r}\|b_p(t)\|_{L^r}
\leq
M_\delta(t)(L\lambda_{p-q})^\delta \lambda_p^{-\delta}.
\]
Since \(\lambda_p=\lambda_q(L\lambda_{p-q})\), the right-hand side equals
\[
M_\delta(t)\lambda_q^{-\delta}.
\]
Because \(\delta>1\) and \(\lambda_q\geq \lambda_0\), we have
\[
\lambda_q^{-\delta}\leq \lambda_0^{1-\delta}\lambda_q^{-1}.
\]
Hence, after increasing \(C\) if needed,
\[
(L\lambda_{p-q})^\delta \lambda_p^{n/r}\|b_p(t)\|_{L^r}<c_r\mu
\qquad \forall p>q
\]
whenever
\[
\lambda_q\geq \frac{C}{\mu}M_\delta(t).
\]

\medskip

\noindent
\emph{Step 2: verification of the low-frequency condition.}
Using Bernstein's inequality term by term,
\[
\|\nabla b_{\le q}(t)\|_{L^\infty}
\leq
\sum_{p\le q}\|\nabla b_p(t)\|_{L^\infty}
\lesssim
\sum_{p\le q}\lambda_p^{1+n/r}\|b_p(t)\|_{L^r}.
\]
By the definition of \(M_\delta(t)\),
\[
\lambda_p^{1+n/r}\|b_p(t)\|_{L^r}
\leq
M_\delta(t)\lambda_p^{1-\delta}.
\]
Since \(\delta>1\), the geometric series
\[
\sum_{p\geq -1}\lambda_p^{1-\delta}
\]
is finite. Consequently,
\[
\|\nabla b_{\le q}(t)\|_{L^\infty}
\lesssim
M_\delta(t),
\]
and therefore
\[
\lambda_q^{-1}\|\nabla b_{\le q}(t)\|_{L^\infty}
\lesssim
M_\delta(t)\lambda_q^{-1}.
\]
After enlarging \(C\) if necessary, this is less than \(c_r\mu\) whenever
\[
\lambda_q\geq \frac{C}{\mu}M_\delta(t).
\]

\medskip

\noindent
\emph{Step 3: conclusion.}
We have shown that every \(q\) with
\[
\lambda_q\geq \frac{C}{\mu}M_\delta(t)
\]
belongs to the admissible set defining \(\Lambda_{b,r}(t)\). Hence
\[
\Lambda_{b,r}(t)
\leq
\frac{C}{\mu}M_\delta(t)
=
\frac{C}{\mu}\|b(t)\|_{B^{\delta+n/r}_{r,\infty}}.
\]
Taking the supremum over \(t\geq 0\) proves the final claim.
\end{proof}

%\section{Preliminaries}

%\bigskip

%\section*{Acknowledgement}
%The author is grateful for ...

%\Endrefs
\end{document}